\documentclass[11pt]{article}
\usepackage[english]{babel}
\usepackage{amssymb,amsmath,amsthm}

\newtheorem{theorem}{Theorem}[section]
\newtheorem{lemma}[theorem]{Lemma}
\newtheorem{proposition}[theorem]{Proposition}
\newtheorem{corollary}[theorem]{Corollary}

{\theoremstyle{definition}}

{\theoremstyle{definition}\newtheorem{definition}[theorem]{Definition}}
{\theoremstyle{definition}\newtheorem{remark}[theorem]{Remark}}

\numberwithin{equation}{section}

\def\C{{\mathbb C}}
\def\N{{\mathbb N}}
\def\Z{{\mathbb Z}}
\def\R{{\mathbb R}}

\def\K{{\mathbb K}}

\def\epsilon{\varepsilon}
\def\kappa{\varkappa}
\def\phi{\varphi}
\def\leq{\leqslant}
\def\geq{\geqslant}

\def\dim{{\rm dim}\,}

\def\ker{\hbox{\tt ker}\,}
\def\im{\hbox{\tt im}\,}

\def\deg{\hbox{\tt deg}\,}

\def\lll{\langle}
\def\rrr{\rangle}
\def\dd{\delta}

\title{Three dimensional Sklyanin algebras and Gr\"obner bases}

\author{Natalia Iyudu, Stanislav Shkarin}

\date{}

\begin{document}

\maketitle

\begin{abstract}We consider a Sklyanin algebra $S$ with 3 generators, which is the quadratic algebra over a field $\K$ with $3$ generators $x,y,z$ given by $3$ relations $pxy+qyx+rzz=0$, $pyz+qzy+rxx=0$ and $pzx+qxz+ryy=0$, where $p,q,r\in\K$. This class of algebras enjoyed much of attention, in particular, using tools from algebraic geometry, Feigin, Odesskii \cite{odf}, and Artin, Tate and Van den Bergh \cite{ATV2}, showed that if at least two of the parameters $p$, $q$ and $r$  are non-zero and at least two of three numbers $p^3$, $q^3$ and $r^3$ are distinct, then $S$
is Koszul and has the same Hilbert series as the algebra of commutative polynomials in 3 variables. %(PHS).

It became commonly accepted, that it is impossible to achieve the same objective by purely algebraic and combinatorial means, like the Gr\"obner basis technique. The main purpose of this paper is to trace the combinatorial meaning of the properties of Sklyanin algebras, such as Koszulity, PBW, PHS, Calabi-Yau, and  to give a new constructive proof of the above facts due to Artin, Tate and Van den Bergh.

Further, we study a wider class of Sklyanin algebras, namely the situation when all parameters of relations could be different. We call them  generalized Sklyanin algebras. We classify up to isomorphism all genralized Sklyanin algebras with the same Hilbert series as commutative polynomials on 3 variables. We show that generalized Slyanin algebras in general position have a Golod-Shafarevich Hilbert series (with exception of the case of field with two elements).
\end{abstract}



\bigskip

\tableofcontents

\medskip

\vskip2truecm

\section{Introduction \label{s1}}\rm

It is well-known that algebras arising in string theory, from the geometry of Calabi-Yau manifolds, that is,  various versions of Calabi-Yau algebras, enjoy the potentiality-like properties. This in essence comes from the symplectic structure on the manifold. The notion of {\it noncommutative potential} was first introduced by Kontsevich in \cite{Ko}. Let $F=\C \lll x_1,\dots,x_n \rrr$, then the quotient vector space

$F_{cyc}=F/[F,F]$ has a simple basis labeled by cyclic words in the alphabet $ x_1,\dots,x_n$. For each $j=1,\dots,n$ in \cite{Ko} was introduced a linear map $\frac{\dd}{\dd x_j}: F_{cyc} \to F$ defined by its action on monomials $\Phi=x_{i_1}\dots x_{i_n}$ by
$$
\frac{\dd \Phi}{\dd x_j}= \sum_{s \vert i_s=j} x_{i_s+1}x_{i_s+2}\dots x_{i_r}x_{i_1}x_{i_2}\dots x_{i_s-1}
$$
So, for any element $\Phi \in F_{cyc}$, which is called a potential, one can define a collection of elements $\frac{\dd \Phi}{\dd x_i}$ for $1\leq i\leq n$. An algebra which has a presentation:
$$
{\cal U} = \C \lll x_1,\dots,x_n \rrr /  \left\{\frac{\dd \Phi}{\dd x_i} \right\}_{1\leq i\leq n}
$$
for some $\Phi \in F_{cyc}$ is called a {\it potential algebra}. This can be generalised to
superpotential algebras, or further generalised to algebras defined by multilinear forms, as in \cite{DV1,DV2}.

It is known for $3$-dimensional Calabi-Yau algebras that they are always derived from a superpotential. But
not all superpotential algebras are Calabi-Yau. This question was studied in details in \cite{DV1,DV2}, \cite{BW} (see also references therein), in \cite{Sol} the conditions on potential which ensure CY have been studied. The most general counterpart of potentiality  and its relation to CY (in one of possible definitions) is considered in \cite[Theorem~3.6.4]{G}.

The simplest example of potential algebras are commutative polynomials. Another important example, which have been extensively studied \cite{AS,skl, ode, odf, ATV1, ATV2, W,S1,S2} are Sklyanin algebras. We are aiming here to demonstrate, that such properties of these algebras as  PBW, PHS, Kosulity, Calabi-Yau could be obtained by constructive, purely combinatorial and algebraic methods, avoiding the power of algebraic geometry demonstrated in \cite{ATV1, ATV2} and later papers continuing this line.

Throughout this paper $\K$ is an arbitrary field, $B$ is a graded algebra, and $B_m$ stands for the $m^{\rm th}$ graded component of algebra $B$. If $V$ is an $n$-dimensional vector space over $\K$, then $F=F(V)$ is the tensor algebra of $V$. For any choice of a basis $x_1,\dots,x_n$ in $V$, $F$ is naturally identified with the free $\K$-algebra with the generators $x_1,\dots,x_n$. For subsets $P_1,\dots,P_k$ of an algebra $B$, $P_1\dots P_k$ stands for the linear span of all products $p_1\dots p_k$ with $p_j\in P_j$. We consider a degree grading on the free algebra $F$: the $m^{\rm th}$ graded component of $F$ is $V^m$. If $R$ is a subspace of the $n^2$-dimensional space $V \otimes V$, then the quotient of $F$ by the ideal $I$ generated by $R$ is called a {\it quadratic algebra} and denoted $A(V,R)$. For any choice of bases $x_1,\dots,x_n$ in $V$ and $g_1,\dots,g_k$ in $R$, $A(V,R)$ is the algebra given by generators $x_1,\dots,x_n$ and the relations $g_1,\dots,g_k$ ($g_j$ are linear combinations of monomials $x_ix_j$ for $1\leq i,j\leq n$). Since each quadratic algebra $A$ is degree graded, we can consider its Hilbert series
$$
H_A(t)=\sum_{j=0}^\infty {\rm dim}_{\K} A_j\,\,t^j.
$$

Quadratic algebras whose Hilbert series is the same as for the algebra $\K[x_1,\dots,x_n]$ of commutative polynomials play a particularly important role in physics. We say that $A$ is a {\it PHS} (for 'polynomial Hilbert series') if
$$
H_A(t)=H_{\K[x_1,\dots,x_n]}(t)=(1-t)^{-n}.
$$

Following the notation from the Polishchuk and Positselski book \cite{popo}, we say that a quadratic algebra $A=A(V,R)$ is a {\it PBW-algebra} (Poincare, Birkhoff, Witt) if there are bases $x_1,\dots,x_n$ and $g_1,\dots,g_m$ in $V$ and $R$ respectively such that with respect to some compatible with multiplication well-ordering on the monomials in $x_1,\dots,x_n$, $g_1,\dots,g_m$ is a (non-commutative) Gr\"obner basis of the ideal $I_A$ generated by $R$. In this case, $x_1,\dots,x_n$ is called a {\it PBW-basis} of $A$, while $g_1,\dots,g_m$ are called the {\it PBW-generators} of $I_A$.

In order to avoid confusion, we would like to stress from the start that Odesskii \cite{ode} as well as some other authors use the term PBW-algebra for what we have already dubbed PHS. Since we deal with both concepts, we could not possibly call them the same and we opted to follow the notation from \cite{popo}.

Another concept playing an important role in this paper is Koszulity. For a quadratic algebra $A=A(V,R)$, the augmentation map $A\to \K$ equips $\K$ with the structure of a commutative graded $A$-bimodule. The algebra $A$ is called {\it Koszul} if $\K$ as a graded right $A$-module has a free resolution $\dots\to M_m\to\dots\to M_1\to A\to\K\to 0$ with the second last arrow being the augmentation map and with each $M_m$ generated in degree $m$. The last property is the same as the condition that the matrices of the above maps $M_m\to M_{m-1}$ with respect to some free bases consist of elements of $V$ (=are homogeneous of degree $1$).

If $(p,q,r)\in \K^3$, the {\it Sklyanin algebra} $Q^{p,q,r}$ is the quadratic algebra over $\K$ with generators $x,y,z$ given by $3$ relations
\begin{equation*}%\label{skl00pqr}
\text{$pyz+qzy+rxx=0$},\quad \text{$pzx+qxz+ryy=0$},\quad\text{$pxy+qyx+rzz=0$}.
\end{equation*}
Note that if $p\neq 0$, then $Q^{p,q,r}$ is obviously the same as the algebra
$S^{a,s}$ with 3 generators is the quadratic algebra over $\K$ with generators $x,y,z$ given by $3$ relations
\begin{equation*}%\label{skl00}
\text{$yz-azy-sxx=0$},\quad \text{$zx-axz-syy=0$},\quad\text{$xy-ayx-szz=0$},
\end{equation*}
where $a=-\frac qp$, $s=-\frac rp$. This way, we reduce the number of parameters, and will deal with algebras $S^{a,s}$. 
%Sklyanin algebras were intensely studied, especially during past two decades, see, for %instance \cite{AS,ATV1,ATV2,ode,odf,S1,S2}. 

Odesskii \cite{ode} proved that in the case $\K=\C$, a generic Sklyanin algebra is a PHS. That is,
\begin{equation*}
\textstyle H_{S^{a,s}}(t)=\sum\limits_{j=0}^\infty \frac{(j+2)(j+1)}{2}\,t^j\ \ \text{for generic $(a,s)\in\C^2$,}
\end{equation*}
where generic means outside the union of countably many  algebraic varieties in $\C^2$ (different from $\C^2$). In particular, the equality  above holds for almost all $(a,s)\in\C^2$ with respect to the 4-dimensional Lebesgue measure. Polishchuk and Positselski \cite{popo} showed in the same setting and with the same meaning of the word 'generic', that for generic $(a,s)\in\C^2$, the algebra $S$ is Koszul but is not a PBW-algebra.

For further references, we label these results:
\begin{equation}\label{opp}
\text{a generic Sklyanin algebra $S^{a,s}$ over $\C$ is  Koszul and PHS.}
\end{equation}

The same results are contained in the Artin, Shelter paper \cite{AS}.

Artin, Tate and Van den Bergh
\cite{ATV1,ATV2}, and Feigin, Odesskii \cite{odf}, considered certain family of infinite dimensional representations of Sklyanin algebra, namely representations, where variables are represented by matrices with one nonzero upper diagonal. In other words, they considered modules with one-dimensional graded components. This was very instructive, and core for most arguments.
They showed that if at least two of the parameters $p$, $q$ and $r$  are non-zero and the equality $p^3=q^3=r^3$ fails, then $Q^{p,q,r}$ is Artin--Shelter regular. More specifically, $Q^{p,q,r}$ is Koszul and has the same Hilbert series as the algebra of commutative polynomials in three variables.

It became commonly accepted that it is impossible to obtain the same results by purely algebraic and combinatorial means like the Gr\"obner basis technique, see, for instance, comments in \cite{ode,W}. The main purpose of this paper is to perform this very impossibility. Namely, we prove the same results by using only combinatorial algebraic techniques, but not algebraic geometry. Mainly, we use just (non-commutative) Gr\"obner basis approach.

\begin{theorem}\label{copo0} The algebra $Q^{p,q,r}$ is Koszul for any $(p,q,r)\in\K^3$. The algebra $Q^{p,q,r}$ is PHS
if and only if at least two of $p$, $q$ and $r$ are non-zero and the equality $p^3=q^3=r^3$ fails.
\end{theorem}

We stress again that the above theorem is essentially one of the main results in \cite{ATV2}. However, our proof is very different.
It is based entirely on Gr\"obner bases computations, properties of Koszul algebras and their Hilbert series, and certain other arguments of cobinatorial nature. This approach is substantially different from the proofs in Artin, Tate, Van den Bergh papers
\cite{ATV1, ATV2}, for example, they get the fact that Sklyanin algebras are PHS as a byproduct of Koszulity. We do it the other way around, we find the Hilbert series first, and then use it to prove Koszulity.

This work was in a way motivated by the question, asked by
Sokolov \cite{suri}, on whether there exist a constructive way to determine, for which paprameters (generalised) Sklyanin algebras are PHS.
Answering this we realized that we can provide a
constructive proofs of known results on Koszulity, PBW and PHS properties of 3-dimensional Sklyanin algebras, due to Artin, Tate, Van den Bergh.
The only results from \cite{ATV1, ATV2}, which we were not able to recover by Gr\"obner bases methods, deals with really subtle question on whether it is a domain. One can feel a taste of the level of difficulty of questions related to zero divisors and nilpotents in rings, algebras, groups, looking at classical papers in this area \cite{Zelm, Ag1,Ag2,Ag3,Ag4,Er}.

To complete the picture we determine which of these algebras are PBW.

\begin{theorem}\label{copo00} The algebra $Q^{p,q,r}$ is PBW if and only if at least one of the following conditions is satisfied$:$
\begin{itemize}\itemsep=-2pt
\item[\rm (\ref{copo00}.1)] $pr=qr=0;$
\item[\rm (\ref{copo00}.2)] $p^3=q^3=r^3;$
\item[\rm (\ref{copo00}.3)] $(p+q)^3+r^3=0$ and the equation $t^2+t+1=0$ is solvable in $\K$.
\end{itemize}
\end{theorem}

The condition of solvability of the quadratic equation above is automatically satisfied if $\K$ is algebraically closed or if $\K$ has characteristic $3$. On the other hand, if $\K=\R$, the third case is empty.

By Theorem~\ref{copo0}, in the case $\K=\C$, there are exactly $10$ pairs $(a,s)$ such that $S^{a,s}$ is not a PHS.
Note that for an arbitrary field $\K$ there no more then 10 cases, which are not PHS. There are no obstacles to the Koszulity of $S$.

We also study the case of generalized Sklyanin algebras, namely we show that if instead of keeping coefficients in the relations to be triples of the same numbers $p,q,r$, we allow them to be all different, the situation changes dramatically. For instance, we show that generically such algebras are finite-dimensional and non-Koszul.

For $q=(a,b,c,\alpha,\beta,\gamma)\in\K^6$, consider the {\it generalized Sklyanin algebra} $\widehat{S}^q$ given by the generators $x$, $y$, $z$ and the relations
\begin{equation}\label{skl0}
\text{$yz-azy-\alpha xx=0$},\quad\text{$zx-bxz-\beta yy=0$},\quad\text{$xy-cyx-\gamma zz=0$}.
\end{equation}

The situation with Koszulity as well as with the generic series for generalized Sklyanin algebras $\widehat{S}^q$ is spectacularly different from that of the Sklyanin algebras $S^{a,s}$.

\begin{theorem}\label{copo2} For $q=(a,b,c,\alpha,\beta,\gamma)$ from a non-empty Zarisski open subset of $\K^6$, $\widehat{S}^q$ is finite dimensional and non-Koszul.
\end{theorem}

By the above result, if $\K$ is infinite, a Zarisski-generic $\widehat{S}^q$ is very far from being a PHS. However, it is possible to figure out exactly which $\widehat{S}^q$ are PHSs.
We give here a complete classification of generalised Sklyanin algebras with respect to the PHS property.

\begin{theorem}\label{copo1} For $q=(a,b,c,\alpha,\beta,\gamma)\in\K^6$, the algebra $\widehat{S}^q$ is a PHS if and only if at least one of the following conditions is satisfied$:$
\begin{itemize}\itemsep=-2pt
\item[\rm (\ref{copo1}.1)] $a=b=c\neq0$ and $(a^3,\alpha\beta\gamma)\neq (-1,-1);$
\item[\rm (\ref{copo1}.2)] $(a,b,c)\neq(0,0,0)$ and either $\alpha=\beta=b-a=0$ OR $\gamma=\alpha=c-a=0$ OR $\beta=\gamma=b-c=0;$
\item[\rm (\ref{copo1}.3)] $a=b=c=0$ and $\alpha\beta\gamma\neq0;$
\item[\rm (\ref{copo1}.4)] $\alpha=\beta=\gamma=0$ and $(a,b,c)\neq(0,0,0);$
\item[\rm (\ref{copo1}.5)] $a^9=-1$, $a^3\neq-1$, $\{b,c\}=\{a^7,a^{13}\}$ and $\alpha\beta\gamma=-a^6$.
\end{itemize}
Furthermore, if $\widehat{S}^q$ is a PHS, then it is Koszul.
\end{theorem}

In the case $\alpha\beta\gamma\neq 0$, where all squares are present, the list shortens considerably.

\begin{corollary}\label{copo101} For $q=(a,b,c,\alpha,\beta,\gamma)\in\K^6$ satisfying $\alpha\beta\gamma\neq 0$, the algebra $\widehat{S}^q$ is a PHS if and only if
either $a=b=c$ and $(a^3,\alpha\beta\gamma)\neq (-1,-1)$ or $(\ref{copo1}.5)$ is satisfied.
\end{corollary}

We recall some known facts on Koszul and PBW algebras and prove few useful technical lemmas in Section~\ref{s2}. We make a number of easy preliminary observations in Section~\ref{observ}. Theorem~\ref{copo0} is proved in Section~\ref{s3}, while Theorem~\ref{copo00} is proved in Section~\ref{s4}. In Section~\ref{s5} we show that the situation changes dramatically if instead of keeping coefficients in the relations to be triples of the same numbers, we allow them to be all different. For instance, we show that generically such algebras are finite dimensional and non-Koszul.

\section{General background \label{s2}}

We shall use the following well-known facts, all of which can be found in \cite{popo}. Every monomial quadratic algebra $A=A(V,R)$ (=there are linear bases $x_1,\dots,x_n$  and $g_1,\dots,g_m$ in $V$ and $R$ respectively, such that each $g_j$ is a monomial in $x_1,\dots,x_n$) is a PBW-algebra. Next, if we pick a basis $x_1,\dots,x_n$ in $V$, we get a bilinear form $b$ on the free algebra $F=F(V)$ defined by $b(u,v)=\delta_{u,v}$ for every monomials $u$ and $v$ in the variables $x_1,\dots,x_n$. The algebra $A^!=A(V,R^\perp)$, where $R^\perp=\{u\in V^2:b(r,u)=0\ \text{for each}\ r\in R\}$, is called the {\it dual algebra} of $A$. Clearly, $A^!$ is a quadratic algebra in its own right. Recall also that there is a specific complex of free right $A$-modules, called the Koszul complex, whose exactness is equivalent to the Koszulity of $A$:
\begin{equation}\label{koco1}
\dots\mathop{\longrightarrow}^{d_{k+1}} (A^!_k)^*\otimes A\mathop{\longrightarrow}^{d_k} (A^!_{k-1})^*\otimes A
\mathop{\longrightarrow}^{d_{k-1}}\dots \mathop{\longrightarrow}^{d_1} (A^!_{0})^*\otimes A=A\longrightarrow \K\to 0,
\end{equation}
where the tensor products are over $\K$, the second last arrow is the augmentation map, each tensor product carries the natural structure of a free right $A$-module and $d_k$ are given by $d_k(\phi \otimes u)=\sum\limits_{j=1}^n \phi_j\otimes x_ju$, where $\phi_j\in (A^!_{k-1})^*$, $\phi_j(v)=\phi(x_jv)$. Although $A^!$ and the Koszul complex seem to depend on the choice of a basis in $V$, it is not really the case up to the natural equivalence \cite[Chapter 2]{popo}. We recall that

\begin{align}\notag%\label{pbk}
&\text{every PBW-algebra is Koszul;}
\\
\notag%\label{stm1}
&\text{$A$ is Koszul $\iff$ $A^!$ is Koszul};
\\
\label{stm2}
&\text{if $A$ is Koszul, then $H_A(-t)H_{A^!}(t)=1$}.
\end{align}

Note that the Koszul complex (\ref{koco1}) of any quadratic algebra is exact at its last $3$ terms: $\K$, $(A^!_0)^*\otimes A=A$ and $(A^!_{1})^*\otimes A$. It follows then that if $H_{A^!}$ is a polynomial of degree $2$, then $A$ is Koszul if and only if $H_A(-t)H_{A^!}(t)=1$ \cite{popo}. That is, the Koszulity of such algebras is determined by their Hilbert series. We generalize this statement
to the case when $H_{A^!}$ is a polynomial of any degree.

\begin{proposition}\label{deg3} Let $A=A(V,R)$ be a quadratic algebra such that $H_{A^!}$ is a polynomial of degree $k$, and Koszul complex of $A$ is exact in all terms, with at most one exception.
Then $A$ is Koszul if and only if $H_A(-t)H_{A^!}(t)=1$. \end{proposition}

\begin{proof}
Excluding trivial cases suppose that $k \geq 3$
Let us denote series of $A$ and $A^!$ respectively:
$$
H_{A^!}(t)=1+nt+dt^2+\sum\limits_{j=3}^k s_jt^j=\sum\limits_{j=0}^k s_jt^j
$$
and
$$
H_{A}(t)=1+nt+(n^2-d)t^2+\sum\limits_{j=3}^k a_jt^j=\sum\limits_{j=0}^\infty a_jt^j.
$$
Consider the Koszul complex:
\begin{equation}\label{koco01}
0\to\dots\mathop{\longrightarrow}^{d_{k+1}} (A^!_k)^*\otimes A\mathop{\longrightarrow}^{d_k} (A^!_{k-1})^*\otimes A
 \mathop{\longrightarrow}^{d_{k-1}}\dots \mathop{\longrightarrow}^{d_1} (A^!_{0})^*\otimes A=A\longrightarrow \K\to 0,
\end{equation}
and its splitting with respect to $A$-grading, namely the corresponding sequence, starting from $l^{\rm th}$ term:
\begin{align}
&0 \to (A^!_k)^*\otimes A_l\mathop{\longrightarrow}^{d_k}
(A^!_{k-1})^*\otimes A_{l+1} \mathop{\longrightarrow}^{d_{k-1}}\dots
\mathop{\longrightarrow}^{d_{m+1}}
(A^!_{m})^*\otimes A_{k+l-m} \mathop{\longrightarrow}^{d_m}\label{kocospli}
 \\
 &(A^!_{m-1})^*\otimes A_{k+l-m+1} \mathop{\longrightarrow}^{d_{m-1}} \dots
 \mathop{\longrightarrow}^{d_2} (A^!_{1})^*\otimes A_{k+l-1}
 \mathop{\longrightarrow}^{d_1} (A^!_{0})^*\otimes A_{k+l} \mathop{\longrightarrow}^{d_0} \K\to 0,\notag
\end{align}
Let the Koszul complex be exact except at the $m^{\rm th}$ term $(A^!_{m})^*\otimes A$.
Now we use the exactness of (\ref{kocospli}) at  terms
$(A^!_k)^*\otimes A_l,\dots, (A^!_{m+1})^*\otimes A_{k+l-m-1}$,
and get the equality:
$$
\dim(\im d_{m+1}\cap (A^!_{m})^*\otimes A_{k+l-m}) = s_{m+1}a_{k+l-1}-s_{m+2}a_{k+l-m-2}+{\dots}+(-1)^{k-m+1}s_ka_l.
$$
 The exactness  at  terms
$(A^!_{m-1})^*\otimes A_{k+l-m+1},\dots, (A^!_{0})^*\otimes A_{k+l}$ gives us:
$$
\dim(\ker d_{m}\cap (A^!_{m})^*\otimes A_{k+l-m}) = s_{m}a_{k+l-m}-s_{m-1}a_{k+l-m+1}+{\dots}+(-1)^{m}s_0a_{k+l}.
$$
The exactness  of the sequences at the $m^{\rm th}$ term $(A^!_{m})^*\otimes A_{k+l-m}$ according to the above expressions for $\im$ and $\ker$ will mean:
$$
\sum\limits_{j=0}^k (-1)^j s_j a_{k+l-j}=0
$$
for all $l$, which is exactly the condition on the series:
$$
H_A(-t)H_{A^!}(t)=1.
$$
\end{proof}

We shall use Proposition~\ref{deg3} in a rather specific situation. To make this application easier, we derive the following corollaries.

\begin{corollary}\label{deg3a0} Let $A=A(V,R)$ be a quadratic algebra such that $A^!_4=\{0\}$
and
\begin{equation}\label{koco2}
0\to (A^!_3)^*\otimes A \mathop{\longrightarrow}^{d_3} (A^!_2)^*\otimes A \mathop{\longrightarrow}^{d_2} (A^!_{1})^*\otimes A\mathop{\longrightarrow}^{d_1} (A^!_{0})^*\otimes A=A\mathop{\longrightarrow}^{d_0} \K\to 0
\end{equation}
be the Koszul complex of $A$. Assume also that $d_3$ is injective. Then $A$ is Koszul if and only if $H_A(-t)H_{A^!}(t)=1$.
\end{corollary}

We say that $u\in A=A(V,R)$ is a {\it right annihilator} if $Vu=\{0\}$ in $A$. A right annihilator $u$ is {\it non-trivial} if $u\neq 0$.

\begin{corollary}\label{deg3a} Let $A=A(V,R)$ be a quadratic algebra such that $A^!_4=\{0\}$, $A^!_{3}$ is one-dimensional and $wA_2^!\neq\{0\}$ for every non-zero $w\in A_1^!$. Then the following statements are equivalent$:$
\begin{itemize}\itemsep=-2pt
\item[\rm (\ref{deg3a}.1)] $A$ is Koszul$;$
\item[\rm (\ref{deg3a}.2)] $A$ has no non-trivial right annihilators and $H_A(-t)H_{A^!}(t)=1$.
\end{itemize}
\end{corollary}

\begin{proof} Fix a basis $x_1,\dots,x_n$ in $V$. Since $A^!_4=\{0\}$ and $A^!_{3}$ is one-dimensional, the Koszul complex of $A$ is of the shape
\begin{equation}\label{koco3}
0\to A=(A^!_3)^*\otimes A \mathop{\longrightarrow}^{d_3} (A^!_2)^*\otimes A \mathop{\longrightarrow}^{d_2} (A^!_{1})^*\otimes A\mathop{\longrightarrow}^{d_1} (A^!_{0})^*\otimes A=A\mathop{\longrightarrow}^{d_0} \K\to 0.
\end{equation}
Let $\phi:A^!_3\to\K$ be the linear isomorphism identifying $(A^!_3)^*\otimes A$ with $\K\otimes A=A$. By definition $d_3:A\to (A^!_2)^*\otimes A$ acts according to the formula $d_3(u)=\sum\limits_{j=1}^n \phi_j\otimes x_ju$, where $\phi_j(v)=\phi(x_jv)$. Clearly, the condition $wA_2^!\neq\{0\}$ for $w\in A^!_1\setminus\{0\}$ yields linear independence of $\phi_1,\dots,\phi_n$ in $(A^!_2)^*$. It follows that $d_3(u)=0$ if and only if $u$ is a right annihilator in $A$. Thus
\begin{equation}\label{ra}
\text{$d_3$ is injective if and only if $A$ has no non-trivial right annihilators.}
\end{equation}
If $A$ is Koszul, the complex (\ref{koco3}) is exact and therefore $d_3$ is injective. By (\ref{ra}), $A$ has no non-trivial right annihilators. Furthermore, $H_A(-t)H_{A^!}(t)=1$ according to (\ref{stm2}). Thus (\ref{deg3a}.1) implies (\ref{deg3a}.2).

Assume now that (\ref{deg3a}.2) is satisfied. By (\ref{ra}), $d_3$ is injective. So we can apply Proposition~\ref{deg3}, and get that $A$ is Koszul. Thus (\ref{deg3a}.2) implies (\ref{deg3a}.1).
\end{proof}

Our next observation is that neither Koszulity nor the Hilbert series of a quadratic algebra $A=A(V,R)$ is sensitive to changing the ground field.

\begin{remark}\label{rem1}Fix the bases $x_1,\dots,x_n$ and $r_1,\dots,r_m$ in $V$ and $R$ respectively. Then $A=A(V,R)$ is given by the generators $x_1,\dots,x_n$ and the relations $r_1,\dots,r_m$. Let $\K_0$ be the subfield of $\K$ generated by the coefficients in the relations $r_1,\dots,r_m$ and $B$ be the $\K_0$-algebra defined by the exact same generators $x_1,\dots,x_n$ and the exact same relations $r_1,\dots,r_m$. Then $A$ is Koszul if and only if $B$ is Koszul (see, for instance, \cite{popo}) and the Hilbert series of $A$ and of $B$ coincide. The latter follows from the fact that the Hilbert series depends only on the set of leading monomials of the Gr\"obner basis. Now the Gr\"obner basis construction algorithm for $A$ and for $B$ produces exactly the same result. Thus if a quadratic algebra given by generators and relations makes sense over 2 fields of the same characteristic, then the choice of the field does not effect its Hilbert series or its Koszulity. In particular, replacing the original field $\K$ by its algebraic closure or by an even bigger field does not change the Hilbert series or Koszulity of $A$. On the other hand, the PBW-property is sensitive to changing the ground field \cite{popo}.
\end{remark}

The next lemma admits a natural generalization to the case of algebras with any number $n$ of generators. We stick with $n=3$ since it is the only case we apply it in.

\begin{lemma}\label{pbbw}
Let $A=A(V,R)$ be a quadratic $\K$-algebra such that $\dim V=\dim R=3$ and $\dim A_3=10$. Then the following hold$:$
\begin{itemize}\itemsep=-2pt
\item[\rm (\ref{pbbw}.1)] If there are linear bases $x$, $y$, $z$ in $V$ and $f$, $g$, $h$ in $R$ and an order $<$ on the monomials compatible with the multiplication such that the leading monomials $\overline{f}$, $\overline{g}$ and $\overline{h}$ of $f$, $g$ and $h$ satisfy
\begin{equation}\label{lemo}
\!\!\!\!\!\!\!\!\!\!\!\!\!\!\{\overline{f},\overline{g},\overline{h}\}\in\bigl\{\{xy,xz,yz\},\{yx,yz,xz\},\{xy,xz,zy\},\{yx,zx,zy\},\{yx,yz,zx\},\{xy,zx,zy\}\bigr\}.
\end{equation}
then $\{x,y,z\}$ is a PBW-basis of $A$ and $f$, $g$, $h$ are PBW-generators of $I_A$. In particular, $A$ is a PBW-algebra and is Koszul. Furthermore, $A$ is a PHS$;$
\item[\rm (\ref{pbbw}.2)] If $A$ is a PBW-algebra with a PBW-basis $\{x,y,z\}$ and PBW-generators $f$, $g$, $h$, then their leading monomials $\overline{f}$, $\overline{g}$ and $\overline{h}$ must satisfy $(\ref{lemo})$.
\end{itemize}
\end{lemma}

\begin{proof} First, suppose that the assumptions of (\ref{pbbw}.1) are satisfied. It is easy to see that there are exactly $10$ degree 3 monomials which do not contain a degree 2 submonomial from
$\{\overline{f},\overline{g},\overline{h}\}$. Furthermore, there is exactly one overlap of the leading monomials $\overline{f}$, $\overline{g}$ and $\overline{h}$. If this overlap produces a non-trivial degree 3 member of the Gr\"obner basis of the ideal $I_A$ of the relations of $A$, we have $\dim A_3=10-1=9$, which violates the assumption $\dim A_3=10$. Hence $f$, $g$ and $h$ form a Gr\"obner basis of $I_A$. Thus $A$ is a PBW-algebra and therefore is Koszul. Now choosing between the left-to-right and the right-to-left degree-lexicographical orderings and ordering the variables appropriately, we can assure that the leading monomials of the standard relations $xy-yx$, $xz-zx$ and $yz-zy$ of $\K[x,y,z]$ are exactly $\overline{f}$, $\overline{g}$ and $\overline{h}$. Since these relations form a Gr\"obner basis of $I_A$, the Hibert series of $A$ and $\K[x,y,z]$ are the same (the Hilbert series depends only on the set of leading monomials of the members of a Gr\"obner basis). Hence $A$ is a PHS. This concludes the proof of (\ref{pbbw}.1).

Now assume that $A$ is a PBW-algebra with a PBW basis $\{x,y,z\}$ and PBW-generators $f$, $g$, $h$. Since $f$, $g$ and $h$ form a Gr\"obner basis of $I_A$, it is easy to see that $\dim A_3$ is $9$ plus the number of overlaps of the leading monomials $\overline{f}$, $\overline{g}$ and $\overline{h}$ of $f$, $g$ and $h$. Since $\dim A_3=10$, the monomials $\overline{f}$, $\overline{g}$ and $\overline{h}$ must produce exactly one overlap. Now it is a straightforward routine check that if at least one of three degree 2 monomials in 3 variables is a square, these monomials overlap at least twice. The same happens, if the three monomials contain $uv$ and $vu$ for some distinct $u,v\in\{x,y,z\}$. Finally, the triples $(uv,vw,wu)$ and $(vu,uw,wv)$ produce 3 overlaps apiece. The only option left is for $\overline{f}$, $\overline{g}$ and $\overline{h}$ to satisfy $(\ref{lemo})$.
\end{proof}

Another tool we use is the following elementary and known fact about the varieties of quadratic algebras. We sketch its proof for the sake of convenience.

\begin{lemma}\label{minhs}
Assume that
\begin{equation}\label{alvar}
\begin{array}{l}
\text{$V$ is an $n$-dimensional vector space over $\K$ and for $1\leq j\leq d$, $q_j:\K^m\to V^2$}\\
\text{is a polynomial map. For each $b\in\K^m$, let $R_b={\rm span}\{q_1(b),\dots,q_d(b)\}$,}\\
\text{which defines the quadratic algebra $A^b=A(V,R_b)$.}
\end{array}
\end{equation}
For $k\in\Z_+$, let
$$
h_k=\min_{b\in\K^m} \dim A_k^b.
$$
Then the non-empty set $\{b\in\K^m:\dim A_k^b=h_k\}$ is Zarissky open in $\K^m$.
\end{lemma}

\begin{proof} We can assume that $k\geq 2$ (for $k\in\{0,1\}$, the set in question is the entire $\K^m$). Pick $c\in\K^m$ such that $\dim A_k^c=h_k$. Denoting $I^b=I_{A^b}$, we then have $\dim I_k^c=n^k-h_k$. Note that since $I^b_k$ is the linear span of $u q_j(b) v$, where $1\leq j\leq d$, $u,v$ are monomials and the degree of $uv$ is $k-2$,
$\dim I_k^b$ is exactly the rank of the rectangular $n^{k-2}d(k-1)\times n^k$ $\K$-matrix $M(b)$ of the coefficients of all $u q_j(b) v$. Let $M_1(b),\dots,M_N(b)$ be all $(n^k-h_k)\times(n^k-h_k)$ submatrices of $M(b)$. For each $j$, let $\delta_j(b)$ be the determinant of the matrix $M_j(b)$. Clearly, each $\delta_j$ is a (commutative) polynomial in the variables $b=(b_1,\dots,b_m)$. Obviously,
$$
G=\{b\in\K^m:\dim A^b_k>h_k\}=\{b\in\K^m:\dim I_k^b<n^k-h_k\}=\{b\in\K^m:\delta_1(b)={\dots}=\delta_N(b)=0\}
$$
is Zarissky closed. Since $\dim A_k^c=h_k$, $c\notin G$ and therefore $G\neq \K^m$. On the other hand, if $b\in U=\K^m\setminus G$, then $\dim A^b_k\leq h_k$. By the definition of $h_k$, $\dim A^b_k\geq h_k$ and therefore $\dim A^b_k=h_k$. Thus $U=\{b\in\K^m:\dim A^b_k=h_k\}$. The required result immediately follows.
\end{proof}

The following result of Drinfeld \cite{dr} features also as Theorem~2.1 in Chapter~6 in \cite{popo}. To explain it properly, we need to remind the characterization of Koszulity in terms of the distributivity of lattices of vector spaces. Let $A=A(V,R)$ be a quadratic algebra. For $n\geq 3$, let $L_n(V,R)$ be the finite lattice of subspaces of $V^n$ generated by the spaces $V^kRV^{n-2-k}$ for $0\leq k\leq n-2$ (as usual, the lattice operations are sum and intersection). Then $A$ is Koszul if and only if $L_n(V,R)$ is distributive for each $n\geq 3$ (see \cite[Chapter~3]{popo}). The mentioned result of Drinfeld is as follows.

\begin{lemma}\label{dri}
Assume that $(\ref{alvar})$ is satisfied and $U$ is a non-empty Zarissky open subset of $\K^m$ such that $\dim A_2^b$ and $\dim A_3^b$ do not depend on $b$ for $b\in U$. Then for each $k\geq 3$, the set
$$
\{b\in U:L_j(V,R_b)\ \ \text{for $3\leq j\leq k$ are distributive}\}
$$
is Zarissky open in $\K^m$.
\end{lemma}

The proof of the above lemma is rather classical. It is a blend of the same argument as in the proof of Lemma~\ref{minhs} with an appropriate inductive procedure. Chiefly, we need the following corollary of Lemmas~\ref{minhs} and~\ref{dri}. Recall that if $\K$ is uncountable, then we say that a {\it generic} $s\in \K^m$ has a property $P$ if $P$ is satisfied for all $s\in\K^m$ outside a union of countably many algebraic varieties (different from whole $\K^n$).

\begin{corollary}\label{hsdri} Assume that $\K$ be uncountable and $(\ref{alvar})$ is satisfied and $h_k=\min\limits_{b\in\K^m} \dim A_k^b$ for $k\in\Z_+$.
Then for generic $b\in\K^m$, $H_{A^b}(t)=\sum\limits_{k=0}^\infty h_kt^k$. Furthermore, exactly one of the following statements holds true$:$
\begin{itemize}\itemsep=-2pt
\item[\rm (\ref{hsdri}.1)] $A^b$ is non-Koszul for every $b\in\K^m$ satisfying $\dim A^b_3=h_3$ and $\dim A^b_2=h_2;$
\item[\rm (\ref{hsdri}.2)] $A^b$ is Koszul for generic $b\in\K^m.$
\end{itemize}
\end{corollary}

\begin{proof} By Lemma~\ref{minhs}, $H_{A^b}(t)=\sum\limits_{k=0}^\infty h_kt^k$ for $b$ from the intersection of countably many non-empty Zarissky open sets and therefore for a generic $b\in\K^m$. By Lemma~\ref{minhs},
$$
U=\{b\in\K^m:\dim A^b_3=h_3\ \text{and}\ \dim A^b_2=h_2\}
$$
is a non-empty Zarissky open subset of $\K^m$. If
$A^b$ is non-Koszul for every $b\in U$, (\ref{hsdri}.1) is satisfied. Assume now that (\ref{hsdri}.1) fails. Then there is $c\in U$ for which $A^c$ is Koszul. By Lemma~\ref{dri}, $W_k=\{b\in U:L_j(V,R_b)\ \ \text{for $3\leq j\leq k$ are distributive}\}$ is Zarissky open in $\K^m$. Since $A^c$ is Koszul, $c\in W_k$ for every $k\geq 3$. Since for $b$ from the intersection of $W_k$ with $k\geq 3$, $A^b$ is Koszul and each $W_k$ is Zarissky open and non-empty, (\ref{hsdri}.2) is satisfied. Obviously, (\ref{hsdri}.1) and (\ref{hsdri}.2) are incompatible.
\end{proof}

\section{Elementary observations \label{observ}}

Obviously, multiplying $(p,q,r)\in\K^3$ by a non-zero scalar does not change the algebra $Q^{p,q,r}$. It turns out that there are non-proportional triples of parameters, which lead to isomorphic (as graded algebras) Sklyanin algebras.

\subsection{Some isomorphisms of Sklyanin algebras}

\begin{lemma}\label{swap}
For every $(p,q,r)\in\K^3$, the graded algebras $Q^{p,q,r}$ and $Q^{q,p,r}$ are isomorphic.
\end{lemma}

\begin{proof} Swapping two of the variables, while leaving the third one as is, provides an isomorphism between $Q^{p,q,r}$ and $Q^{q,p,r}$.
\end{proof}

\begin{lemma}\label{root1} Assume that $(p,q,r)\in\K^3$ and $\theta\in\K$ is such that $\theta^3=1$ and $\theta\neq 1$. Then the graded algebras $Q^{p,q,r}$ and $Q^{p,q,\theta r}$ are isomorphic.
\end{lemma}

\begin{proof} The relations of $Q^{p,q,r}$ in the variables $u$, $v$, $w$ given by $x=u$, $y=v$ and $z=\theta^2w$ read
$puv+qvu+\theta rww=0$, $pwu+quw+\theta rvv=0$ and $pvw+qwv+\theta ruu=0$. Thus this change of variables provides an isomorphism between $Q^{p,q,r}$ and $Q^{p,q,\theta r}$.
\end{proof}

\begin{lemma}\label{root2} Assume that $(p,q,r)\in\K^3$ and $\theta\in\K$ is such that $\theta^3=1$ and $\theta\neq 1$. Then the graded algebras $Q^{p,q,r}$ and $Q^{p',q',r'}$ are isomorphic, where $p'=\theta^2p+\theta q+r$, $q'=\theta p+\theta^2q+r$ and $r'=p+q+r$.
\end{lemma}

\begin{proof} A direct computation shows that the space of the quadratic relations of $Q^{p,q,r}$ in the variables $u$, $v$, $w$ given by $x=u+v+w$, $y=u+\theta v+\theta^2 w$ and $z=u+\theta^2v+\theta w$ (the matrix of this change of variables is non-degenerate) is spanned by $p'uv+q'vu+r'ww=0$, $p'wu+q'uw+r'vv=0$ and $p'vw+q'wv+r'uu=0$. Thus $Q^{p,q,r}$ and $Q^{p',q',r'}$ are isomorphic.
\end{proof}

\subsection{Easy degenerate cases}

First, if $p=q=r=0$, then $Q^{p,q,r}$ is the free algebra and therefore $A=Q^{p,q,r}$ is PBW and therefore Koszul and has the Hilbert series $H_A(t)=(1-3t)^{-1}$. If exactly two of $p$, $q$ and $r$ are $0$, then $A$ is monomial and therefore is PBW and therefore Koszul. One easily verifies that in this case $H_A(t)=\frac{1+t}{1-2t}$. If $p^3=q^3=r^3\neq 0$, one easily checks that the defining relations of $A$ form a Gr\"obner basis in the ideal they generate. Hence $A$ is PBW and therefore Koszul. Furthermore, the Hilbert series of $A$ is the same as for the monomial algebra given by the leading monomials $xx$, $xy$ and $xz$  of the relations of $A$. It follows that again $H_A(t)=\frac{1+t}{1-2t}$. If $r=0$ and $pq\neq 0$, Lemma~\ref{pbbw} yields that $A$ is PBW (and therefore Koszul) PHS. The latter means that $H_A=(1-t)^{-3}$. As a matter of fact, $A$ in this case is the algebra of quantum polynomials. These observations are summarised in the following lemma.

\begin{lemma}\label{dege1} The Sklyanin algebra $A=Q^{p,q,r}$ is PBW and therefore is Koszul if $r=0$, or if $p=q=0$, or if $p^3=q^3=r^3$. Moreover, $H_A(t)=(1-3t)^{-1}$ if $p=q=r=0$, $H_A(t)=\frac{1+t}{1-2t}$ if exactly two of $p$, $q$ and $r$ are $0$ or if $p^3=q^3=r^3\neq 0$ and $H_A=(1-t)^{-3}$ if $r=0$ and $pq\neq 0$.
\end{lemma}

\subsection{The Hilbert series of the dual algebra}

\begin{lemma} \label{duser} Let $(p,q,r)\in\K$ and $A=Q^{p,q,r}$. Then the Hilbert series of $A^!$ is given by
\begin{equation}\label{hsaa}
H_{A^!}(t)=\left\{\begin{array}{ll}1+3t&\text{if $p=q=r=0$;}
\\ \textstyle\frac{1+2t}{1-t}&\text{if $p^3=q^3=r^3\neq 0$ or exactly two of $p$, $q$ and $r$ equal $0$;}\\ (1+t)^{3}&\text{otherwise}.
\end{array}\right.
\end{equation}
Moreover, $wA_2^!\neq\{0\}$ for each non-zero $w\in A_1^!$ provided $H_{A^!}(t)=(1+t)^{3}$.
\end{lemma}

\begin{proof} If $p=q=r=0$, the result is trivial. If $p^3=q^3=r^3\neq 0$ or exactly two of $p$, $q$ and $r$ equal $0$, Lemma~\ref{dege1} yields that $A$ is Koszul and $H_A(t)=\frac{1+t}{1-2t}$. By (\ref{stm2}), $H_{A^!}(t)=\frac{1+2t}{1-t}$.
If $r=0$ and $pq\neq 0$, then Lemma~\ref{dege1} yields that $A$ is Koszul and $H_A(t)=(1-t)^{-3}$. By (\ref{stm2}), $H_{A^!}(t)=(1+t)^3$. Thus (\ref{hsaa}) holds if $r=0$ or $p^3=q^3=r^3$ or $p=q=0$.

Now consider the case $r\neq 0$, $(p,q)\neq (0,0)$ and $pq=0$. By Lemma~\ref{swap}, $A$ is isomorphic to $S^{0,s}$ for some $s\neq 0$. The defining relations of $A^!$ in this case can be written as $yx=0$, $xz=0$, $zy=0$, $xy=-\frac1szz$, $yy=-szx$ and $xx=-syz$. Applying the non-commutative Buchberger algorithm, we get that the (finite) Gr\"obner basis of the ideal $I_{A^!}$ of the relations of $A^!$ is
\begin{align*}
&\textstyle yx,\quad xz,\quad zy, \quad xy+\frac1szz,\quad yy+szx,\quad xx+syz,
\\
&\textstyle yzx+\frac1szzz,\quad zzx,\quad yzz,\quad zzzx,\quad yzzz,\quad zzzz.
\end{align*}
Then the only normal words are $x$, $y$, $z$, $zx$, $yz$, $zz$ and $zzz$ and therefore $H_{A^!}(t)=1+3t+3t^2+t^3=(1+t)^3$, which proves (\ref{hsaa}) in the case $r\neq 0$, $(p,q)\neq (0,0)$ and $pq=0$.

Thus it remains to consider the case when $pqr\neq 0$ and $p^3=q^3=r^3$ fails. In this case $A$ is isomorphic to $S^{a,s}$
with $as\neq 0$ and $(a^3,s^3)\neq(-1,-1)$. The defining relations of $A^!$ then can be written as $xx=\frac sazy$, $xy=-\frac1szz$, $yx=\frac aszz$, $yy=-szx$, $xz=-azx$ and $yz=-\frac1azy$. A direct computation shows that
\begin{align*}
&\textstyle xx-\frac sazy,\quad xy+\frac1szz,\quad yx-\frac aszz,\quad yy+szx,
\\
&\textstyle xz+azx,\quad yz+\frac1azy,\quad zzy,\quad zzx,\quad zzzz
\end{align*}
is a Gr\"obner basis of $I_{A^!}$. The only normal words are $x$, $y$, $z$, $zx$, $zy$, $zz$ and $zzz$. Again, we have $H_{A^!}(t)=1+3t+3t^2+t^3=(1+t)^3$, which completes the proof of (\ref{hsaa}).

Assume now that $H_{A^!}(t)=(1+t)^3$ and $w=\alpha x+\beta y+\gamma z$ be a non-zero element of $A_1^!=V$. It remains to show that $wA_2^!\neq\{0\}$. If $r=0$, (\ref{hsaa}) yields $pq\neq 0$. Then $A=S^{a,0}$ with $a\neq 0$. It is easy to see that the one-dimensional space $A_3^!$ is spanned by $yzx=zxy=xyz$ and that every monomial with at least two copies of the same letter vanishes in $A^!$. Then for $g=\alpha_1yz+\beta_1zx+\gamma_1xy$ with $\alpha_1,\beta_1,\gamma_1\in \K$, we have $wg=(\alpha\alpha_1+\beta\beta_1+\gamma\gamma_1)yzx$. Since $(\alpha,\beta,\gamma)\neq (0,0,0)$, it follows that $wA_2^!\neq\{0\}$. If $r\neq 0$, from the above description of the Gr\"obner basis of $I_{A^!}$ it follows that the one-dimensional space $A_3^!$ is spanned by $xxx=yyy=zzz$ and that every monomial of degree $3$ with exactly  two copies of the same letter (like $xxy$ or $zyz$) vanishes in $A^!$. Then for $g=\alpha_1xx+\beta_1yy+\gamma_1zz$ with $\alpha_1,\beta_1,\gamma_1\in \K$, we have $wg=(\alpha\alpha_1+\beta\beta_1+\gamma\gamma_1)zzz$. Since $(\alpha,\beta,\gamma)\neq (0,0,0)$, it follows that $wA_2^!\neq\{0\}$.
\end{proof}

Note \cite{popo} that for every quadratic algebra $A=A(V,R)$ (Koszul or otherwise), the power series $H_A(t)H_{A^!}(-t)-1$ starts with $t^k$ with $k\geq 4$. This allows to determine $\dim A_3$ provided we know $\dim A^!_j$ for $j\leq 3$. Applying this observation together with (\ref{hsaa}), we immediately obtain the following fact.

\begin{corollary} \label{a3} Let $(p,q,r)\in\K$ and $A=Q^{p,q,r}$. Then
\begin{equation}\label{hsaa1}
\dim A_3=\left\{\begin{array}{ll}27&\text{if $p=q=r=0;$}
\\ 12&\text{if $p^3=q^3=r^3\neq 0$ or exactly two of $p$, $q$ and $r$ equal $0;$}\\ 10&\text{otherwise}.
\end{array}\right.
\end{equation}
\end{corollary}

\subsection{Lower estimate for $H_{Q^{p,q,r}}$}

\begin{lemma} \label{loest}
For every $(p,q,r)\in\K$, $\dim Q^{p,q,r}_n\geq \frac{(n+1)(n+2)}{2}$ for every $n\in\Z_+$.
\end{lemma}

\begin{proof} By Remark~\ref{rem1}, we can without loss of generality assume that $\K$ is uncountable (just replace $\K$ by an uncountable field extension, if necessary). For each $n\in\Z_+$, let
$$
d_n=\min_{(a,b,c)\in\K^3}\dim Q^{a,b,c}_n.
$$
Clearly, $d_2=6$. By (\ref{hsaa1}), $d_3=10$. Obviously, $P=Q^{1,-1,0}=\K[x,y,z]$ is Koszul and $\dim P_2=6=d_2$, $\dim P_3=10=d_3$. By Corollary~\ref{hsdri} and Lemma~\ref{duser}, for generic $(a,b,c)\in\K^3$, $A=Q^{a,b,c}$ is Koszul and satisfies $H_{A}(t)=\sum\limits_{n=0}^\infty d_nt^n$ and $H_{A^!}(t)=(1+t)^3$. Now by (\ref{stm2}), $\sum\limits_{n=0}^\infty d_nt^n=(1-t)^{-3}$ and therefore $d_n=\frac{(n+1)(n+2)}{2}$ for every $n\in\Z_+$. Now the result follows from the definition of $d_n$.
\end{proof}

\section{Proof of Theorem~\ref{copo0} \label{s3}}

Throughout this section $p,q,r\in \K$ and $A=Q^{p,q,r}$. If $p^3=q^3=r^3$ or $p=q=0$ or $r=0$, the conclusion of Theorem~\ref{copo0} follows from Lemma~\ref{dege1}. We split our consideration into cases. First, we eliminate easier ones.

\subsection{Case $pq=0$, $(pr,qr)\neq(0,0)$}

By Lemma~\ref{swap}, we can without loss of generality assume that $p\neq 0$ and $q=0$. Since $r\neq 0$, $A=S^{0,s}$ for some $s\neq 0$. It turns out that in this case, the Gr\"obner basis of the ideal $I_A$ of the relations of $A$ is
$$
\begin{array}{l}
xx-\frac1syz,\ xy-szz,\ yy-\frac1szx,\ xzx-s^2zzy,\ xzz-\frac1{s^2}yzy,\\
yzx-szzz,\ xzyz-s^3zzyx,\ yzyz-s^2zzzx,\ yzzz-zzzy.
\end{array}
$$
None of the leading monomials of the members of this basis starts with $z$. It follows that the set of normal words is closed under multiplication by $z$ from the left. Hence $zu\neq 0$ for every non-zero $u\in A$ and therefore $A$ has no non-trivial right annihilators.

Since the set of leading monomials depends neither on $s$ nor on the underlying field $\K$, we have $H_A=H_B$, where $B=S^{0,1}=Q^{1,0,-1}$ is a  $\C$-algebra. Let $\theta=e^{2\pi i/3}\in \C$. Using Lemma~\ref{root2}, we see that $B$ is isomorphic to $Q^{1,-\theta,0}$. By Lemma~\ref{dege1}, $H_B(t)=(1-t)^{-3}$ and therefore $H_A(t)=(1-t)^{-3}$. By Lemma~\ref{duser} and Corollary~\ref{deg3a} $A$ is Koszul, which completes the proof of Theorem~\ref{copo0} in this case.

\subsection{Case $p^3=q^3\neq 0$, $r\neq 0$}

In this case $A=S^{a,s}$ with $a^3=-1$ and $s\neq 0$ and the Gr\"obner basis of the ideal of the relations of $A$ is
$$
\textstyle xx-\frac1syz+\frac aszy,\ xy-ayx-szz,\ xz-\frac1a zx-\frac sa yy,\ yyz+\frac1a zyy,\ yzz+\frac1a zzy.
$$
None of the leading monomials of the members of this basis starts with $z$. As above, it follows that $zu\neq 0$ for every non-zero $u\in A$ and therefore $A$ has no non-trivial right annihilators.

It is easy to describe the normal words. Namely, they are the words of the shape
$z^k(yz)^ly^mx^\epsilon$ with $k,l,m\in\Z_+$ and $\epsilon\in\{0,1\}$. Now one easily sees that the number of normal words of degree $n$ is exactly the number of pairs $(k,m)$ of non-negative integers satisfying $k+m\leq n$, which is $\frac{(n+1)(n+2)}2$. Indeed, for every $k,m\in\Z_+$ satisfying $k+m\leq n$, there are unique $l\in\Z_+$ and $\epsilon\in\{0,1\}$ for which the degree of $z^k(yz)^ly^mx^\epsilon$ is $n$. Hence $H_A(t)=(1-t)^{-3}$. By Lemma~\ref{duser} and Corollary~\ref{deg3a} $A$ is Koszul, which completes the proof of Theorem~\ref{copo0} in this case.

\subsection{Case $(p^3-r^3)(q^3-r^3)=0$, $(p^3-r^3,q^3-r^3)\neq (0,0)$ and $pqr\neq 0$}

By Lemma~\ref{swap}, we can without loss of generality assume that $p^3=r^3$. Now by Lemma~\ref{root1}, we can without loss of generality assume that $p=r$. Hence $A=S^{a,-1}$, where $a\neq 0$ and $a^3\neq -1$. Unfortunately, in this case the Gr\"obner basis of the ideal of the relations of $A$ does not appear to be finite. However there is a way around that. Namely, computing the Gr\"obner basis of the ideal of the relations of $A=S^{a,-1}$ up to degree 4 (there are 2 elements of degree 3 and 2 elements of degree 4), one  easily verifies that $g=yzx-zzz$, is  central in $A$. Now let $B=A/I$, where $I$ is the ideal in $A$ generated by $g$. The Gr\"obner basis of the ideal of the ideal of the relations of $B$ is
$$
\begin{array}{l}
x^2{+}yz{-}azy,\ xy{-}ayx{+}z^2,\ xz{-}\frac1azx{-}\frac1ay^2,\ yzx{-}z^3,\ y^2x{-}zyz,\ y^2z{-}ayzy{+}\frac1azy^2{+}\frac1az^2x,
\\
yzyz{-}ayz^2y{+}z^3x,\ yzyx{-}\frac1ayz^3{-}\frac1az^3y,\ yzy^2{+}yz^2x{-}az^4,\ y^4{+}yz^3{-}azyz^2,%\hhhh
\\
yz^2yx{-}\frac1ayz^4{+}\frac1az^3yz{-}z^4y,\ yz^2y^2{-}z^3yx,\ yz^3y{-}z^4x,\ yz^2yzy{-}\frac1{a^2}yz^4x{-}\frac1{a^2}z^3y^3{+}\frac1az^4yx,%\hhhh
\\
yz^2yz^2{+}\frac1ayz^4y{-}\frac1az^3yzy{-}z^4x,\ yz^6{-}z^6y,\ yz^4yz{-}z^5y^2,\ yz^4y^2{+}yz^5x{-}a^2z^5yx{+}az^7.%\hhhh
\end{array}
$$
Yet again, none of the leading monomials of the members of this basis starts with $z$. Hence $zu\neq 0$ in $B$ for every non-zero $u\in B$. Note that the set of leading monomials depends neither on $a$ nor on the underlying field $\K$. Let $C$ be the algebra $A$ in the case $\K=\C$ and $a=2$ and $D$ be the corresponding algebra $B$: $D=C/\langle g\rangle$. Since $C=Q^{1,-2,1}$, Lemma~\ref{root2} yields that $C$ is isomorphic to $Q^{1,\theta,0}$, where $\theta=2^{2\pi i/3}$. In particular, $C$ being isomorphic to the algebra of quantum polynomials in 3 variables has no zero divisors and satisfies $H_C(t)=(1-t)^{-3}$. Hence $D$, being a factor of $C$ by a central element of degree $3$, satisfies $\dim D_0=1$, $\dim D_1=3$, $\dim D_2=6$ and $\dim D_n=\dim C_n-\dim C_{n-3}=3n$ for $n\geq 3$. Since $H_B=H_D$, we have
$H_B(t)=1+\sum\limits_{n=1}^\infty 3nt$. Now since $B$ is a factor of $A$ by a central element of degree $3$, we have
$\dim A_{n}\leq \dim A_{n-3}+\dim B_{n}=\dim A_{n-3}+3n$ for $n\geq 3$ and all these inequalities turn into equalities precisely when $g$ is not a zero divisor. Solving these recurrent inequalities and using the initial data $\dim A_0=1$, $\dim A_1=3$, $\dim A_2=6$, we get $\dim A_n\leq \frac{(n+1)(n+2)}2$ for $n\in\Z_+$ and all these inequalities turn into equalities precisely when $g$ is not a zero divisor. Combining this with Lemma~\ref{loest}, we conclude that $H_A=(1-t)^{-3}$ and that $g$ is not a zero divisor in $A$.

Now assume that there is a non-zero homogeneous element of $A$ satisfying $zu=0$. Then there is such an element $u$ of the lowest degree. Since $zu=0$ in $B$, we have $u=0$ in $B$. By definition of $B$, there is $v\in A$ such that $u=vg$ in $A$. Then $zvg=0$ in $A$. Since $g$ is not a zero divisor $zv=0$ in $A$. Since $v$ is non-zero and has degree lower (by 3) than $u$, we have arrived to a contradiction. Hence $zu\neq 0$ in $A$ for every non-zero $u\in A$ and therefore $A$ has no non-trivial right annihilators. By Lemma~\ref{duser} and Corollary~\ref{deg3a} $A$ is Koszul, which completes the proof of Theorem~\ref{copo0} in this case.

\subsection{Main case $pqr(p^3-r^3)(q^3-r^3)(p^3-q^3)\neq 0$}

In this case $A=S^{a,s}$ with $as(a^3+1)(s^3+1)(a^3-s^3)\neq 0$. For the sake of brevity, we use the following notation
$$
\alpha=a^3+1\ \ \text{and}\ \ \beta=s^3+1.
$$
The above restrictions on $a$ and $s$ yield $\alpha\beta(\alpha-\beta)(\alpha-1)(\beta-1)\neq 0$. In this case
$$
\textstyle g=yyy+\frac{\alpha-\beta}{s\beta}yzx-\frac{a}{s}zyx+\frac{\alpha-\beta}{\beta}zzz
$$
is a non-zero central element in $A$. It is given in \cite{AS} and reproduced in \cite{ATV2}. In fact it is straightforward (we have done it to be on the safe side) to verify that $g$ is indeed non-zero and central by computing the members of the Gr\"obner basis of the ideal of the relations of $A$ up to degree $4$. Now we consider the algebra
$$
B=A/I,\ \ \text{where $I$ is the ideal in $A$ generated by $g$.}
$$
In other words, $B$ is given by the generators $x$, $y$ and $z$ and the relations
\begin{align}
xx&=\textstyle \frac1syz-\frac aszy, \label{RE1}
\\
xy&=\textstyle ayx+szz, \label{RE2}
\\
yz&=\textstyle syy+azy, \label{RE3}
\\
yyy&=\textstyle -\frac{\alpha-\beta}{s\beta}yzx+\frac{a}{s}zyx-\frac{\alpha-\beta}{\beta}zzz, \label{RE4}
\end{align}
where the first three of the above relations are the defining relations of $A$. Resolving the overlaps $xxy$, $xxz$ and $yyxz$, we obtain further $3$  relations holding in $B$:
\begin{align}
yyx&=\textstyle -\frac{a^2\beta}{s^2\alpha}yzz+\frac{1}{s}zyz-\frac{a\beta}{s^2\alpha}zzy, \label{RE5}
\\
yyz&=\textstyle \frac{a\alpha}{\alpha-\beta}yzy-\frac{1}{a}zyy-\frac{s^2\alpha}{a(\alpha-\beta)}zzx, \label{RE6}
\\
\textstyle \frac{\alpha^2-\alpha\beta+\beta^2}{\beta(\alpha-\beta)}yzyx&=\textstyle\frac{s(\alpha^2-\alpha\beta+\beta^2-\alpha^2\beta)}
{a\alpha\beta(\beta-1)}yzzz+\frac{a(\alpha^2-\alpha\beta+\beta^2)}{s^2\alpha(\alpha-\beta)}zzyz +\frac{\alpha^2-\alpha\beta+\beta^2-\alpha^2\beta}{as^2\beta(\alpha-\beta)}zzzy. \label{RE7}
\end{align}
Note that (\ref{RE4}), (\ref{RE5}) and (\ref{RE6}) correspond to all degree $3$ members of the Gr\"obner basis for the ideal of the relations of $B$, while (\ref{RE7}) is just one degree $4$ member of the same basis.

Next we consider the graded right $B$-module
$$
M=B/zB.
$$
The reason for doing this is apparent from the following lemma.

\begin{lemma}\label{whyM} The following implications hold true
\begin{align}
&\textstyle H_M(t)=1+2t+\sum\limits_{n=2}^\infty 3t^n\,\,\Longrightarrow\,\, H_A(t)=(1-t)^{-3}\ \text{and $A$ is Koszul;} \label{MA1}
\\
&\text{$n\geq 2$ and $\dim M_j\leq 3$ for $2\leq j\leq n$}\,\,\Longrightarrow\,\,\dim M_{j}=3\ \text{for}\ 2\leq j\leq n.
\label{MA2}
\end{align}
\end{lemma}

\begin{proof} Clearly,
$$
\dim B_j=\dim zB_{j-1}+\dim B_j/zB_{j-1}=\dim zB_{j-1}+\dim M_j\ \ \text{for $j\geq1$.}
$$
Hence
\begin{equation*}%\label{MA3}
\begin{array}{l}
\dim B_j\leq \dim B_{j-1}+\dim M_j\ \text{for $j\geq 1$;}\\
\dim B_j=\dim B_{j-1}+\dim M_j\iff zu\neq 0\ \ \text{for $u\in B_{j-1}\setminus\{0\}$}.
\end{array}
\end{equation*}

Since, obviously, $\dim M_0=\dim B_0=1$ and $\dim M_1=2$, the above display yields
\begin{equation}\label{MA4}
\begin{array}{l}
\text{provided $n\geq 2$ and $\dim M_j\leq 3$ for $2\leq j\leq n$, we have}\\
\dim B_j\leq 3j\ \text{for $1\leq j\leq n$;}\\
\dim B_j=3j\ \text{for $1\leq j\leq n$}\iff \left\{\begin{array}{l}\text{$\dim M_j=3$ for $2\leq j\leq n$ and}\\
\text{$zu\neq 0$ in $B$ for $u\in B\setminus\{0\}$ with $\deg u<n$.}\end{array}\right.
\end{array}
\end{equation}

Since $g$ is central in $A$ and is a homogeneous element of degree $3$, we have
$$
\dim A_j=\dim gA_{j-3}+\dim A_j/gA_{j-3}=\dim gA_{j-3}+\dim B_j\ \ \text{for $j\geq3$.}
$$

Since $\dim A_0=\dim B_0=1$, $\dim A_1=\dim B_1=3$ and $\dim A_2=\dim B_2=6$, the above display yields
\begin{equation}\label{MA5}
\begin{array}{l}
\text{provided $n\geq 3$ and $\dim B_j\leq 3j$ for $1\leq j\leq n$, we have}\\
\dim A_j\leq \frac{(j+1)(j+2)}{2}\ \text{for $0\leq j\leq n$;}\\
\dim A_j=\frac{(j+1)(j+2)}{2}\ \text{for $0\leq j\leq n$}{\iff} \left\{\begin{array}{l}\text{$\dim B_j=3j$ for $1\leq j\leq n$ and}\\
\text{$gu\neq 0$ in $A$ for $u\in A\setminus\{0\}$ with $\deg u\leq$\,\rlap{$n-3$.}}\end{array}\right.
\end{array}
\end{equation}

Combining (\ref{MA4}) and (\ref{MA5}), we get
\begin{equation}\label{MA6}
\begin{array}{l}
\text{provided $n\geq 3$ and $\dim M_j\leq 3$ for $2\leq j\leq n$, we have}\\
\dim A_j\leq \frac{(j+1)(j+2)}{2}\ \text{for $0\leq j\leq n$;}\\
\dim A_j=\frac{(j+1)(j+2)}{2}\ \text{for $0\leq j\leq n$}{\iff} \left\{\begin{array}{l}\text{$\dim M_j=3$ for $2\leq j\leq n$,}\\
\text{$zu\neq0$ in $B$ for $u\in B\setminus\{0\}$ with $\deg u<n$,}\\ \text{$gu\neq 0$ in $A$ for $u\in A\setminus\{0\}$ with \rlap{$\deg u\leq n-3$.}}\end{array}\right.
\end{array}
\end{equation}

On the other hand, be Lemma~\ref{loest}, $\dim A_j\geq \frac{(j+1)(j+2)}{2}$ for each $j\in\Z_+$. Thus (\ref{MA6}) can be rewritten as follows:
\begin{equation}\label{MA7}
\begin{array}{l}
\text{provided $n\geq 3$ and $\dim M_j\leq 3$ for $2\leq j\leq n$, we have}\\
\dim A_j= \frac{(j+1)(j+2)}{2}\ \text{for $0\leq j\leq n$, $\dim M_j=3$ for $2\leq j\leq n$},\\
\text{$zu\neq0$ in $B$ for $u\in B\setminus\{0\}$ with $\deg u<n$}\\
\text{and $gu\neq 0$ in $A$ for $u\in A\setminus\{0\}$ with $\deg u\leq n-3$.}
\end{array}
\end{equation}

Obviously, (\ref{MA2}) is a direct consequence of (\ref{MA7}). Now assume that $H_M(t)=1+2t+\sum\limits_{n=2}^\infty 3t^n$. By (\ref{MA7}), $H_A(t)=(1-t)^{-3}$, $zu\neq 0$ in $B$ for every $u\in B\setminus\{0\}$ and $g$ is not a  zero divisor in $A$. Now we shall show that $zu\neq 0$ for every $u\in A\setminus \{0\}$. Assume the contrary. Then there is the minimal $n\in\N$ for which there exists $u\in A_n\setminus\{0\}$ satisfying $zu=0$ in $A$. Hence $zu=0$ in $B$. Since we already know that $z$ is not a left zero divisor in $B$, $u=0$ in $B$. Hence there is $v\in A$ such that $u=vg$ in $A$. Since $u\neq 0$ in $A$, we have $v\neq 0$ in $A$. Since $0=zu=zvg$ in $A$ and $g$ is not a zero divisor in $A$, we have $zv=0$ in $A$. Since $\deg v=\deg u -3=n-3<n$, we have arrived to a contradiction with the minimality of $n$. Thus $zu\neq 0$  for each $u\in A\setminus\{0\}$ and therefore $A$ has no non-trivial right annihilators. By Lemma~\ref{duser}, $H_{A^!}(t)=(1+t)^3$. Hence $H_A(t)H_{A^!}(-t)=1$. Now Corollary~\ref{deg3a} implies that $A$ is Koszul, which completes the proof.
\end{proof}

According to Lemma~\ref{whyM}, the proof of Theorem~\ref{copo0} will be complete as soon as we prove that $H_M(t)=1+2t+\sum\limits_{n=2}^\infty 3t^n$. The rest of this section is devoted to doing exactly this by means of applying the Gr\"obner basis technique. The second part of Lemma~\ref{whyM} is just a tool which spares us from doing some of the calculations. We start by describing the typical situation in which the components of $M$ find themselves.

For $n\in\Z_+$, we say that condition $\Omega(n)$ is satisfied if
\begin{align}
&\qquad \text{$\dim M_j=3$ for $2\leq j\leq n+3$, $yz^{n+1}V=M_{n+3}$ and there are $p_n,q_n,r_n\in\K$ such that}\notag
\\
&\textstyle yz^nyx=-\frac{a^2}{s^2}p_nyz^{n+2},\ \ yz^nyy=-\frac{1}{s}q_nyz^{n+1}x,\ \ yz^nyz=ar_nyz^{n+1}y, \label{omega1}
\end{align}
where (\ref{omega1}) consists of equalities in $M$.

First, observe that if $\Omega(n)$ is satisfied, $yz^{n+2}$, $yz^{n+1}x$ and $yz^{n+1}y$ are linearly independent in $M$ and therefore the numbers $p_n$, $q_n$ and $r_n$ are uniquely determined. Next, using (\ref{RE4}), (\ref{RE5}) and (\ref{RE6}), one easily sees that
\begin{equation} \label{omeg0}
\text{$\Omega(0)$ is satisfied with $\textstyle p_0=\frac{\beta}{\alpha}$, $\textstyle q_0=\frac{\alpha-\beta}{\beta}$ and $\textstyle r_0=\frac{\alpha}{\alpha-\beta}$.}
\end{equation}

\begin{lemma}\label{omeg1} Assume that $\Omega(n)$ is satisfied. Then the following equations hold in $M:$
\begin{equation}\label{omeg2}
\begin{array}{lll}
\textstyle b_n^{1,1}yz^{n+1}yx=-\frac{a^2}{s^2}c_n^{1,1}yz^{n+3}& \text{and}&\textstyle b_n^{1,2}yz^{n+1}yx=-\frac{a^2}{s^2}c_n^{1,2}yz^{n+3},\\
\textstyle b_n^{2,1}yz^{n+1}yy=-\frac{1}{s}c_n^{2,1}yz^{n+2}x& \text{and}&\textstyle b_n^{2,2}yz^{n+1}yy=-\frac{1}{s}c_n^{2,2}yz^{n+2}x,%\hhhh
\\
b_n^{3,1}yz^{n+1}yz=ac_n^{3,1}yz^{n+2}y& \text{and}&b_n^{3,2}yz^{n+1}yz=ac_n^{3,2}yz^{n+2}y,%\hhhh
\end{array}
\end{equation}
where
\begin{equation}\label{omeg3}
\begin{array}{ll}
b_n^{1,1}=\alpha(\alpha-1)r_n-\beta(\alpha-1),&c_n^{1,1}=\beta(\alpha-1)p_n+(\beta-1)(\alpha-\beta),\\
b_n^{1,2}=\beta(\alpha-1)q_n-(\alpha-1)(\alpha-\beta)r_n+(\alpha-1)\beta,& c_n^{1,2}=\beta(\beta-1)q_n-(\beta-1)(\alpha-\beta),\\
b_n^{2,1}=\beta(\alpha-1)r_n-(\alpha-\beta),&c_n^{2,1}=(\alpha-1)(\alpha-\beta)p_n-\alpha(\beta-1),\\
b_n^{2,2}=(\alpha-\beta)q_n-\alpha(\alpha-1)r_n+(\alpha-\beta),& c_n^{2,2}=-(\alpha-\beta)q_n+\alpha(\beta-1),\\
b_n^{3,1}=(\alpha-\beta)r_n-\alpha,&c_n^{3,1}=\alpha p_n-\beta,\\
b_n^{3,2}=\alpha q_n-\beta(\alpha-1)r_n+\alpha,& c_n^{3,2}=\alpha q_n+\beta.
\end{array}
\end{equation}
Moreover, $(b_n^{2,1},c_n^{2,1},b_n^{2,2},c_n^{2,2})\neq (0,0,0,0)$ and $(b_n^{3,1},c_n^{3,1},b_n^{3,2},c_n^{3,2})\neq (0,0,0,0)$. Furthermore, if $(b_n^{1,1},b_n^{1,2})\neq (0,0)$, $(b_n^{2,1},b_n^{2,2})\neq (0,0)$ and $(b_n^{3,1},b_n^{3,2})\neq (0,0)$, then $\Omega(n+1)$ is satisfied.
\end{lemma}

\begin{proof} The equalities (\ref{omeg2}) are obtained by resolving (and reducing) the overlaps $(yz^kyx)z=yz^ky(xz)$, $(yz^kyy)y=yz^k(yyy)$, $(yz^kyx)x=yz^ky(xx)$, $(yz^kyy)z=yz^k(yyz)$, $(yz^kyx)y=yz^ky(xy)$ and $(yz^kyy)x=yz^k(yyx)$ respectively using (\ref{omega1}) and (\ref{RE1}--\ref{RE6}).

Now, let us show that $(b_n^{2,1},c_n^{2,1},b_n^{2,2},c_n^{2,2})\neq (0,0,0,0)$. Assume the contrary: $b_n^{2,1}=c_n^{2,1}=b_n^{2,2}=c_n^{2,2}=0$. According to (\ref{omeg3}), these equalities yield $p_n=\frac\beta\alpha$, $q_n=\frac{\alpha-\beta}{\beta}$, $r_n=\frac{\alpha-\beta}{\beta(\alpha-1)}$ and $\alpha^2-\alpha\beta+\beta^2-\alpha\beta^2=0$, which together with (\ref{omeg3}) imply that $c_n^{1,2}=c_{n}^{3,1}=0$, $b_n^{3,1}=\frac{\alpha(\beta-\alpha)}{\alpha-1}\neq 0$, $c_n^{3,2}=\alpha\beta\neq 0$, $c_n^{1,1}=-\frac{\beta(\alpha-\beta)}{\alpha}\neq 0$ and $b_{n}^{1,2}=-\alpha(\alpha-\beta)\neq 0$ (recall that $\alpha\beta(\alpha-\beta)(\alpha-1)(\beta-1)\neq 0$). Hence the two equations in the first line of (\ref{omeg2}) are linearly independent and so are the two equations in the third line of (\ref{omeg2}). Thus (\ref{omeg2}) yields
$yz^{n+1}yx=yz^{n+3}=yz^{n+1}yz=yz^{n+2}y=0$ in $M$. Since $M_{n+3}$ is spanned by $yz^{n+1}x$, $yz^{n+1}y$ and $yz^{n+2}$, these equalities imply that $M_{n+4}$ is spanned by $yz^{n+1}yy$ and $yz^{n+2}x$. Hence $\dim M_{n+4}<3$, while $\dim M_j\leq 3$ for $j\leq n+3$. We have arrived to a contradiction with (\ref{MA2}), which proves that $(b_n^{2,1},c_n^{2,1},b_n^{2,2},c_n^{2,2})\neq (0,0,0,0)$.

Next, let us show that $(b_n^{3,1},c_n^{3,1},b_n^{3,2},c_n^{3,2})\neq (0,0,0,0)$. Assume the contrary: $b_n^{3,1}=c_n^{3,1}=b_n^{3,2}=c_n^{3,2}=0$. According to (\ref{omeg3}), these equalities yield $p_n=-q_n=\frac\beta\alpha$, $r_n=\frac\alpha{\alpha-\beta}$ and $\alpha^2-\alpha\beta+\beta^2-\alpha^2\beta=0$, which together with (\ref{omeg3}) imply that $c_n^{1,1}=b_{n}^{2,1}=0$, $c_n^{2,1}=\beta(\alpha-\beta)\neq 0$, $b_n^{1,1}=\frac{\alpha^2\beta(\alpha-1)}{\alpha-\beta}\neq 0$ and $c_n^{1,2}=-\alpha\beta(\beta-1)\neq 0$. Since
$c_n^{1,1}=0$, $b_n^{1,1}\neq 0$ and $c_n^{1,2}\neq 0$, the two equations in the first line of (\ref{omeg2}) are linearly independent. This together with $b_n^{2,1}=0$ and $c_n^{2,1}\neq 0$ implies that $yz^{n+1}yx=yz^{n+2}x=yz^{n+3}=0$ in $M$.

These equalities together with the fact that $M_{n+3}$ is spanned by $yz^{n+1}x$, $yz^{n+1}y$ and $yz^{n+2}$ implies that $M_{n+4}$ is spanned by $yz^{n+1}yy$, $yz^{n+1}yz$ and $yz^{n+2}y$. Resolving the overlaps $(yz^{n+2}x)x=yz^{n+2}(xx)$,
$(yz^{n+2}x)y=yz^{n+2}(xy)$, $(yz^{n+2}x)z=yz^{n+2}(xz)$, $(yz^{n+1}yx)x=yz^{n+1}y(xx)$, $(yz^{n+1}yx)y=yz^{n+1}y(xy)$ and $(yz^{n+1}yx)z=yz^{n+1}y(xz)$ by means of the relations $yz^{n+1}yx=yz^{n+2}x=yz^{n+3}=0$ in $M$ and (\ref{RE1}--\ref{RE6}) in $B$ we get, respectively, that the equalities $yz^{n+2}yz=0$, $yz^{n+2}yx=0$, $yz^{n+2}yy=0$, $yz^{n+1}yzy=0$, $yz^{n+1}yzz=0$ and $yz^{n+1}yzx=0$ are satisfied in $M$. These equalities together with the fact that $M_{n+4}$ is spanned by $yz^{n+1}yy$, $yz^{n+1}yz$ and $yz^{n+2}y$ yield $M_{n+5}=\{0\}$. Again, we have arrived to a contradiction with (\ref{MA2}), which proves that $(b_n^{3,1},c_n^{3,1},b_n^{3,2},c_n^{3,2})\neq (0,0,0,0)$.

Finally, assume that $(b_n^{1,1},b_n^{1,2})\neq (0,0)$, $(b_n^{2,1},b_n^{2,2})\neq (0,0)$ and $(b_n^{3,1},b_n^{3,2})\neq (0,0)$. Then (\ref{omeg3}) yields the existence of $p_{n+1}$, $q_{n+1}$ and $r_{n+1}$ in $\K$ such that (\ref{omega1}) with $n$ replaced by $n+1$ is satisfied. By $\Omega(n)$, $yz^{n+1}V=M_{n+3}$. Hence $yz^{n+1}V^2=M_{n+4}$. Using (\ref{RE1}--\ref{RE3}) and (\ref{omega1}) with $n$ replaced by $n+1$, one easily sees that $yz^{n+2}V=M_{n+4}$. In particular, $\dim M_{n+4}\leq 3$ and therefore $\dim M_{n+4}=3$ by (\ref{MA2}). Hence $\Omega(n+1)$ is satisfied.
\end{proof}

\begin{lemma}\label{periodic} Assume that $\Omega(n)$ is satisfied and $p_n=-q_n=r_n=\frac{\beta}{\alpha}$. Then $H_M(t)=1+2t+\sum\limits_{n=2}^\infty 3t^n$.
\end{lemma}

\begin{proof} It is easy to check that in the case $p_n=-q_n=r_n=\frac{\beta}{\alpha}$, the equations (\ref{omeg2}) provided by Lemma~\ref{omeg1} read $yz^{n+3}=0$, $yz^{n+1}y=0$ and $yz^{n+1}yy=\frac1s yz^{n+2}x$ (in $M$). It follows that $M_{n+4}$ is spanned by $yz^{n+1}yx$, $yz^{n+2}x$ and $yz^{n+2}y$. Now using the relations (\ref{RE1}--\ref{RE6}), it is easy to verify that $M_{n+5}=M_{n+4}V$ is spanned by $yz^{n+2}yx$, $yz^{n+2}yy$ and $yz^{n+2}yz$. That is, $M_{n+5}=yz^{n+2}yV$. Since $yz^{n+3}=0$ in $M$, it follows that if $u\in B$ and $yu=0$ in $M$, then $yz^{n+2}yu=0$ in $M$. Applying this observation to $u\in B_k$ and using the equality $M_{n+4+k}=yz^{n+2}yB_k$ (follows from $M_{n+5}=yz^{n+2}yV$), we get $\dim M_{n+4+k}\leq \dim M_{4+k}$ for $k\in\N$. Since $\Omega(n)$ is satisfied and since we have already checked that $M_{n+4}$ and $M_{n+5}$ have $3$-element spanning sets, we get $\dim M_j\leq 3$ for $j\leq n+5$. Now the inequality $\dim M_{n+4+k}\leq \dim M_{4+k}$ for $k\in\N$ yields $\dim M_j\leq 3$ for all $j$. By (\ref{MA2}), $H_M(t)=1+2t+\sum\limits_{n=2}^\infty 3t^n$.
\end{proof}

\begin{lemma}\label{periodic11} Assume that $\alpha^2-\alpha\beta+\beta^2=0$. Then $H_M(t)=1+2t+\sum\limits_{n=2}^\infty 3t^n$.
\end{lemma}

\begin{proof} By (\ref{omeg0}), $\Omega(0)$ is satisfied with $p_0=\frac{\beta}{\alpha}$, $q_0=\frac{\alpha-\beta}{\beta}$ and $r_0=\frac{\alpha}{\alpha-\beta}$. Using $\alpha^2-\alpha\beta+\beta^2=0$, we see that $p_0=-q_0=r_0=\frac{\beta}{\alpha}$. It remains to apply Lemma~\ref{periodic}.
\end{proof}

\begin{lemma}\label{curve} Assume that $\alpha^2-\alpha\beta+\beta^2\neq 0$ and $\Omega(n)$ is satisfied. Then
\begin{equation}\label{curve1}
p_n(q_n+1)=r_n((\alpha-1)p_n-(\beta-1))=q_n(r_n-1).
\end{equation}
\end{lemma}

\begin{proof} Let $b_n^{j,k}$ and $c_n^{j,k}$ be the numbers defined in (\ref{omeg3}). By Lemma~\ref{omeg1}, $(b_n^{2,1},c_n^{2,1},b_n^{2,2},c_n^{2,2})\neq (0,0,0,0)$ and $(b_n^{3,1},c_n^{3,1},b_n^{3,2},c_n^{3,2})\neq (0,0,0,0)$. Furthermore, the equality $b_n^{1,1}=c_n^{1,1}=b_n^{1,2}=c_n^{1,2}=0$ implies $\alpha^2-\alpha\beta+\beta^2=0$ and therefore $(b_n^{1,1},c_n^{1,1},b_n^{1,2},c_n^{1,2})\neq (0,0,0,0)$. Thus each of the lines in (\ref{omeg2}) contains at least one non-trivial equation. It is a matter of straightforward verification that if in any of the lines the two equations are linearly independent, then $\dim M_{n+4}<3$ and we arrive to a contradiction with (\ref{MA2}). Thus each of the matrices
$$
\left(\begin{array}{cc}
b_n^{1,1}&c_n^{1,1}\\
b_n^{1,2}&c_n^{1,2}
\end{array}\right),\quad \left(\begin{array}{cc}
b_n^{2,1}&c_n^{2,1}\\
b_n^{2,2}&c_n^{2,2}
\end{array}\right)\ \ \text{and}\ \
\left(\begin{array}{cc}
b_n^{3,1}&c_n^{3,1}\\
b_n^{3,2}&c_n^{3,2}
\end{array}\right)
$$
is degenerate. Hence the determinants of the matrices in the above display equal $0$. Plugging in the explicit expressions (\ref{omeg3}) for $b_n^{j,k}$ and $c_n^{j,k}$ and simplifying, we arrive to the system
$$
\begin{array}{l}
0=\alpha^2p_n(q_n+1)-\alpha\beta r_n((\alpha-1)p_n-(\beta-1))-\alpha(\alpha-\beta)q_n(r_n-1);
\\
0=(\alpha-\beta)p_n(q_n+1)-\alpha r_n((\alpha-1)p_n-(\beta-1))+\beta q_n(r_n-1);%\hhhh
\\
0=\beta (\alpha-1)p_n(q_n+1)-(\alpha-\beta)r_n((\alpha-1)p_n-(\beta-1))-\alpha(\beta-1)q_n(r_n-1).%\hhhh
\end{array}
$$
This is a system of linear equations on the variables $p_n(q_n+1)$, $r_n((\alpha-1)p_n-(\beta-1))$ and $q_n(r_n-1)$. The third equation is always a linear combination of the first two, while the first two equations are linearly independent precisely when $\alpha^2-\alpha\beta+\beta^2\neq 0$. Now it is easy to see that this system is equivalent to (\ref{curve1}).
\end{proof}

\begin{lemma}\label{OME} Assume that $\Omega(n)$ is satisfied, $\alpha^2-\alpha\beta+\beta^2\neq0$ and
$$
(p_n,q_n,r_n)\notin \textstyle\Bigl\{\Bigl(\frac\beta\alpha,-\frac\beta\alpha,\frac\beta\alpha\Bigr),
\Bigl(-\frac{(\beta-1)(\alpha-\beta)}{\beta(\alpha-1)},\frac{\alpha-\beta}{\beta},\frac{\alpha-\beta}{\beta(\alpha-1)}\Bigr),
\Bigl(\frac{\alpha(\beta-1)}{(\alpha-1)(\alpha-\beta)},\frac{\alpha(\beta-1)}{\alpha-\beta},\frac\alpha{\alpha-\beta}\Bigr)
\Bigr\}.
$$
Then $\Omega(n+1)$ is satisfied.
\end{lemma}

\begin{proof} Let $b_n^{j,k}$ and $c_n^{j,k}$ be the numbers defined in (\ref{omeg3}). Using the equation (\ref{curve1}) provided by  Lemma~\ref{curve} together with (\ref{omeg3}), we easily obtain that
\begin{align*}
b_n^{1,1}=b_n^{1,2}=0&\iff \textstyle (p_n,q_n,r_n)=\Bigl(\frac\beta\alpha,-\frac\beta\alpha,\frac\beta\alpha\Bigr);
\\
b_n^{2,1}=b_n^{2,2}=0&\iff \textstyle (p_n,q_n,r_n)=
\Bigl(-\frac{(\beta-1)(\alpha-\beta)}{\beta(\alpha-1)},\frac{\alpha-\beta}{\beta},\frac{\alpha-\beta}{\beta(\alpha-1)}\Bigr);
\\
b_n^{3,1}=b_n^{3,2}=0&\iff \textstyle (p_n,q_n,r_n)=
\Bigl(\frac{\alpha(\beta-1)}{(\alpha-1)(\alpha-\beta)},\frac{\alpha(\beta-1)}{\alpha-\beta},\frac\alpha{\alpha-\beta}\Bigr).
\end{align*}

By Lemma~\ref{omeg1}, $\Omega(n+1)$ is satisfied if $(b_n^{1,1},b_n^{1,2})\neq (0,0)$, $(b_n^{2,1},b_n^{2,2})\neq (0,0)$ and $(b_n^{3,1},b_n^{3,2})\neq (0,0)$. Hence the above display yields the required result.
\end{proof}

\begin{lemma}\label{OME1} Assume that $\Omega(n)$ is satisfied, $\alpha^2-\alpha\beta+\beta^2\neq0$ and
$$
(p_n,q_n,r_n)=\textstyle
\Bigl(-\frac{(\beta-1)(\alpha-\beta)}{\beta(\alpha-1)},\frac{\alpha-\beta}{\beta},\frac{\alpha-\beta}{\beta(\alpha-1)}\Bigr).
$$
Then $\Omega(n+2)$ is satisfied.
\end{lemma}

\begin{proof} Plugging $p_n=-\frac{(\beta-1)(\alpha-\beta)}{\beta(\alpha-1)}$, $q_n=\frac{\alpha-\beta}{\beta}$ and $r_n=\frac{\alpha-\beta}{\beta(\alpha-1)}$ into Lemma~\ref{omeg1}, we see that the equations (\ref{omeg2}) read $yz^{n+1}yx=0$, $yz^{n+2}x=0$ and $yz^{n+1}yz=ayz^{n+2}y$ in $M$. It follows that $M_{n+4}$ is spanned by $yz^{n+3}$, $yz^{n+1}yy$ and $yz^{n+2}y$. Using the equation $yz^{n+1}yx=0$ together with (\ref{RE1}--\ref{RE3}), we can resolve  the overlaps $(yz^{n+2}x)x=yz^{n+2}(xx)$, $(yz^{n+2}x)y=yz^{n+2}(xy)$ and $(yz^{n+2}x)z=yz^{n+2}(xz)$ to obtain that
$yz^{n+2}yx=-\frac{s}{a}yz^{n+4}$, $yz^{n+2}yy=\frac{1}{s}yz^{n+3}x$ and $yz^{n+2}yz=ayz^{n+3}y$. It also follows that $M_{n+5}$ is spanned by $yz^{n+4}$, $yz^{n+3}y$ and $yz^{n+3}x$. By (\ref{MA2}), $M_{n+4}$ and $M_{n+5}$ are 3-dimensional. Thus $\Omega(n+2)$ is satisfied with $p_{n+2}=\frac{s^3}{a^3}=\frac{\beta-1}{\alpha-1}$, $q_{n+2}=-1$ and $r_{n+2}=1$.
\end{proof}

\begin{lemma}\label{OME2} Assume that $\Omega(n)$ is satisfied, $\alpha^2-\alpha\beta+\beta^2\neq0$ and
$$
(p_n,q_n,r_n)=\textstyle
\Bigl(\frac{\alpha(\beta-1)}{(\alpha-1)(\alpha-\beta)},\frac{\alpha(\beta-1)}{\alpha-\beta},\frac\alpha{\alpha-\beta}\Bigr).
$$
Then $\Omega(n+3)$ is satisfied.
\end{lemma}

\begin{proof} Plugging $p_n=\frac{\alpha(\beta-1)}{(\alpha-1)(\alpha-\beta)}$, $q_n=\frac{\alpha(\beta-1)}{\alpha-\beta}$ and $r_n=\frac\alpha{\alpha-\beta}$ into Lemma~\ref{omeg1}, we see that the equations (\ref{omeg2}) read $yz^{n+1}yx=-\frac sayz^{n+3}$, $yz^{n+1}yy=0$ and $yz^{n+2}y=0$. It follows that $M_{n+4}$ is spanned by $yz^{n+3}$, $yz^{n+2}x$ and $yz^{n+1}yz$. Using the equation $yz^{n+1}yy=0$ together with (\ref{RE1}--\ref{RE6}), we can resolve  the overlaps $(yz^{n+1}yy)x=yz^{n+1}(yyx)$, $(yz^{n+1}yy)y=yz^{n+1}(yyy)$ and $(yz^{n+1}yy)z=yz^{n+1}(yyz)$ to obtain that
$yz^{n+1}yzz=-\frac{1}{a}yz^{n+3}y$, $yz^{n+1}yzx=-syz^{n+4}$ and $yz^{n+1}yzy=\frac{s^2}{a^2}yz^{n+3}x$. Now
$M_{n+4}$ is spanned by $yz^{n+4}$, $yz^{n+3}y$ and $yz^{n+3}x$. On the next step we resolve the overlaps $(yz^{n+1}yzx)z=yz^{n+1}yz(xz)$, $(yz^{n+1}yzx)x=yz^{n+1}yz(xx)$ and $(yz^{n+1}yzy)x=yz^{n+1}(yzyx)$ with the help of
(\ref{RE1}--\ref{RE7}) and the above equations in $M$ (note that (\ref{RE7}) is needed to resolve $(yz^{n+1}yzy)x=yz^{n+1}(yzyx)$ and that it can be used because $\alpha^2-\alpha\beta+\beta^2\neq0$) to obtain respectively that $yz^{n+3}yx=-\frac{a^2}{s^2}p_{n+3}yz^{n+5}$, $yz^{n+3}yy=-\frac{1}{s}aq_{n+3}yz^{n+4}x$ and $yz^{n+3}yz=ar_{n+3}yz^{n+4}y$ with $p_{n+3}=-\frac{(\beta-1)(\alpha-\beta)}{(\alpha-1)\beta}$, $q_{n+1}=\frac{\alpha(\beta-1)}{\alpha-\beta}$ and $r_n=\frac{\beta}{\alpha}$. It also follows that $M_{n+6}$ is spanned by $yz^{n+5}$, $yz^{n+4}y$ and $yz^{n+4}x$. By (\ref{MA2}), $M_{n+4}$, $M_{n+5}$ and $M_{n+6}$ are 3-dimensional. Thus $\Omega(n+3)$ is satisfied.
\end{proof}

\begin{lemma}\label{finalL} The Hilbert series of $M$ is given by $H_M(t)=1+2t+\sum\limits_{n=2}^\infty 3t^n$.
\end{lemma}

\begin{proof} If $\alpha^2-\alpha\beta+\beta^2=0$, the result is provided by Lemma~\ref{periodic11}. For the rest of the proof we shall assume that $\alpha^2-\alpha\beta+\beta^2\neq 0$. If
$$
\text{there exists $n\in\Z_+$ such that $\Omega(n)$ is satisfied and $\textstyle p_n=-q_n=r_n=\frac\beta\alpha$,}
$$
the result is provided by Lemma~\ref{periodic}. Thus for the rest of the proof we can assume that the condition in the above display fails. Now by Lemmas~\ref{OME}, \ref{OME1} and~\ref{OME2}, if $\Omega(n)$  is satisfied, then $\Omega(m)$ is satisfied for some $m\in\{n+1,n+2,n+3\}$. By (\ref{omeg0}), $\Omega(0)$ is satisfied. Hence $\Omega(n)$ is satisfied for infinitely many $n$. It follows that $\dim M_j=3$ for $j\geq 2$. Since $\dim M_0=1$ and $\dim M_1=2$, we have $H_M(t)=1+2t+\sum\limits_{n=2}^\infty 3t^n$.
\end{proof}

Direct application of Lemmas~\ref{finalL} and~\ref{whyM} conclude the proof of Theorem~\ref{copo0}.

\section{Proof of Theorem~\ref{copo00}\label{s4}}

We need the following elementary fact.

\begin{lemma}\label{abc} Assume that the equation $t^2+t+1=0$ has no solutions in $\K$ and $p,q,r\in\K$ satisfy $p^2+q^2+r^2=pr+qr+pq$. Then $p=q=r$.
\end{lemma}

\begin{proof} The equality $p^2+q^2+r^2=pr+qr+pq$ can be rewritten as $(p-q)^2+(q-r)^2=(p-q)(r-q)$. Assume that $p=q=r$ fails. Then either $p-q\neq 0$ or $r-q\neq 0$. Without loss of generality, we can assume that $p-q\neq 0$. Then the equality $(p-q)^2+(r-q)^2=(p-q)(r-q)$ implies that $t^2+t+1=0$ for $t=\frac{q-r}{p-q}$. We have arrived to a contradiction.
\end{proof}

The next lemma deals with necessary conditions for $S^{a,s}$ to be PBW.

\begin{lemma}\label{nonpbw} Assume that $a,s\in\K$ are such that $s\neq 0$, $(a^3,s^3)\neq(-1,-1)$ and $A=S^{a,s}$ is PBW. Then $(1-a)^3=s^3$ and the equation $t^2+t+1=0$ has a solution in $\K$.
\end{lemma}

\begin{proof} Pick a PBW basis $u,v,w$ in $V$ for $A$ and the corresponding PBW-generators $f,g,h\in R$. Let $\overline{f}$, $\overline{g}$ and $\overline{h}$ be the leading (with respect to the corresponding order $>$) monomials of $f$, $g$ and $h$. Without loss of generality, we may assume that $u>v>w$ and $\overline{f}>\overline{g}>\overline{h}$. By (\ref{hsaa1}), $\dim A_3=10$. By the second part of Lemma~\ref{pbbw},
\begin{equation}\label{lemo1}
\begin{array}{ll}
\text{$\overline{h}\in\{vw,wv\}$ and}& \{\overline{f},\overline{g}\}\in\bigl\{\{uv,uw\},\{vu,uw\},\{vu,wu\}\bigr\}\ \ \text{if $\overline{h}=vw$,}\\
&\{\overline{f},\overline{g}\}\in\bigl\{\{uv,uw\},\{vu,wu\},\{uv,wu\}\bigr\}\ \ \text{if $\overline{h}=wv$.}
\end{array}
\end{equation}

Since there is no degree 2 monomials greater than $uu$, $uu$ does not feature at all in any of $f$, $g$ or $h$. Since $f$, $g$ and $h$ span $R$, $uu$ does not feature in any element of $R$. In particular it does not feature in the original relations $r_1=yz-azy-sxx$, $r_2=zx-axz-syy$ and $r_3=xy-ayx-szz$, when written in terms of the variables $u,v,w$. Since $u$, $v$ and $w$ form a basis in $V$, there are unique $t_1,t_2,t_3\in\K$ such that $x\in t_1u+L$, $y\in t_2u+L$ and $z\in t_3u+L$, where $L$ is the linear span of $v$ and $w$. Since $x$, $y$ and $z$ form a basis of $V$ as well, $(t_1,t_2,t_3)\neq (0,0,0)$. Plugging this data into the definition of $r_j$ we see that the $uu$-coefficients in $r_1$, $r_2$ and $r_3$ (when written in terms of $u$, $v$ and $w$) are $(1-a)t_2t_3-st_1^2$, $(1-a)t_1t_3-st_2^2$ and $(1-a)t_1t_2-st_3^2$ respectively. On the other hand, we know that $r_1$, $r_2$ and $r_3$ do not contain $uu$. Hence $(1-a)t_2t_3-st_1^2=(1-a)t_1t_3-st_2^2=(1-a)t_1t_2-st_3^2=0$. If $t_1=0$, we get $st_2^2=st_3^2=0$ and therefore $t_2=t_3=0$ (recall that $s\neq 0$). This is not possible since $(t_1,t_2,t_3)\neq (0,0,0)$. Thus $t_1\neq 0$. Similarly, $t_2\neq 0$ and $t_3\neq 0$. Multiplying the equalities $(1-a)t_2t_3=st_1^2$, $(1-a)t_1t_3=st_2^2$ and $(1-a)t_1t_2=st_3^2$, we get $(1-a)^3(t_1t_2t_3)^2=s^3(t_1t_2t_3)^2$. Since $t_1t_2t_3\neq 0$, it follows that $(1-a)^3=s^3$.

It remains to show that the equation $t^2+t+1=0$ has a solution in $\K$. This certainly happens if $\K$ has characteristic $3$. Thus for the rest of the proof we can assume that the characteristic of $\K$ is different from $3$. Assume the contrary: there is no $t\in\K$ such that $t^2+t+1=0$. Since $t^3-1=(t-1)(t^2+t+1)$, $1$ is the only solution of the equation $t^3=1$. Since $(1-a)^3=s^3$, it follows that $s=1-a$. Since $s\neq 0$ the equalities $(1-a)t_2t_3-st_1^2=(1-a)t_1t_3-st_2^2=(1-a)t_1t_2-st_3^2=0$ yield $t_2t_3-t_1^2=t_1t_3-t_2^2=t_1t_2-t_3^2=0$. Since $t_j$ are non-zero, from $t_2t_3=t_1^2$, we get $t_3=\frac{t_1^2}{t_2}$. Plugging this into $t_1t_3=t_2^2$, we obtain $t_1^3=t_2^3$ and therefore $t_1=t_2$. Similarly, $t_2=t_3$. Thus $t_1=t_2=t_3\neq 0$. Then without loss of generality, we may assume that $t_1=t_2=t_3=1$. The expressions for $x$, $y$ and $z$ in terms of $u$, $v$ and $w$ now look like $x=u+pv+\alpha w$, $y=u+qv+\beta w$ and $z=u+rv+\gamma w$, where the coefficients are from $\K$. Since both $\{x,y,z\}$ and $\{u,v,w\}$ are linear bases in $V$,
\begin{equation}\label{node}
\text{the matrix}\ \ C=\left(\begin{array}{ccc}
1&p&\alpha\\
1&q&\beta\\
1&r&\gamma
\end{array}\right)\ \ \text{is invertible.}
\end{equation}
 By (\ref{lemo1}), $\overline{h}\in\{vw,wv\}$. Since each of the monomials $uv$, $vu$ and $vv$ is greater than each of $vw$ and $wv$, $h$ should not contain $uv$, $vu$ and $vv$. Since $r_1$, $r_2$ and $r_3$ form a basis in $R$, $h$ is a non-trivial linear combination of $r_1$, $r_2$ and $r_3$. It follows that the $3\times 3$ matrix $M$ of the coefficients in front of $uv$, $vu$ and $vv$ in $r_1$, $r_2$ and $r_3$ written in terms of $u$, $v$ and $w$ must be non-invertible. Plugging $x=u+pv+\alpha w$, $y=u+qv+\beta$ and $z=u+rv+\gamma w$ into $r_1=yz-azy-(1-a)xx$, $r_2=zx-axz-(1-a)yy$ and $r_3=xy-ayx-(1-a)zz$ (recall that $s=1-a$), we easily compute this matrix and then its determinant:
\begin{align*}
M&=\left(\begin{array}{ccc}
q-ap+(a-1)r&p-aq+(a-1)r&(1-a)(pq-r^2)\\
p-ar+(a-1)q&r-ap+(a-1)q&(1-a)(pr-q^2)\\
r-aq+(a-1)p&q-ar+(a-1)p&(1-a)(qr-p^2)
\end{array}\right)\ \ \text{and}
\\
{\rm det}\,M&=(a-1)^2(a+1)(p^2+q^2+r^2-pq-pr-qr)^2.
\end{align*}
Since ${\rm det}\,M=0$ and we know that $a\neq 1$ (otherwise $s=0$), we have that either $a=-1$ or $p^2+q^2+r^2=pq+pr+qr$. By (\ref{node}), the equality $p=q=r$ fails. If $p^2+q^2+r^2=pq+pr+qr$, Lemma~\ref{abc} implies then that the equation $t^2+t+1=0$ has a solution in $\K$.

It remains to consider the case $a=-1$. Then $s=1-a=2$. Since $s\neq 0$, ${\rm char}\,\K\neq 2$. We have $r_1=yz+zy-2xx$, $r_2=zx+xz-2yy$ and $r_3=xy+yx-2zz$ and therefore $r_j$ are symmetric. Since a linear change of variables does not break the symmetry, $r_j$ remain symmetric when written in terms of $u$, $v$ and $w$. It follows that $f$, $g$ and $h$, being linear combinations of $r_j$, are symmetric as well. Since $\overline{h}\in\{vw,wv\}$ and $uv>vv>vw$, $uw>vw$, $vu>vv>wv$ and $wu>wv$, $h$ does not contain either $uv$, $uw$ and $vv$ or $vu$, $wu$ and $vv$. Since $h$ is symmetric, it does not contain $uv$, $uw$ and $vv$ in any case. Since $h$ is a non-trivial linear combination of $r_j$, it follows that the $3\times 3$ matrix $N$ of the coefficients in front of $uv$, $uw$ and $vv$ in $r_1$, $r_2$ and $r_3$ written in terms of $u$, $v$ and $w$ must be non-invertible. Plugging $x=u+pv+\alpha w$, $y=u+qv+\beta$ and $z=u+rv+\gamma w$ into $r_1=yz+zy-2xx$, $r_2=zx+xz-2yy$ and $r_3=xy+yx-2zz$,  we easily verify that
$$
N=\left(\begin{array}{ccc}
p+q-2r&\alpha+\beta-2\gamma&2(pq-r^2)\\
p+r-2q&\alpha+\gamma-2\beta&2(pr-q^2)\\
q+r-2p&\beta+\gamma-2\alpha&2(qr-p^2)
\end{array}\right),\quad {\rm det}\,N=6(p^2+q^2+r^2-pq-pr-qr)\,{\rm det}\,C,
$$
where $C$ is the matrix defined in (\ref{node}). Since the characteristic of $\K$ is neither 2 nor 3, $C$ is invertible and $N$ is non-invertible, it follows that $p^2+q^2+r^2=pq+pr+qr$. As above, an application of Lemma~\ref{abc} yields that the equation $t^2+t+1=0$ has a solution in $\K$.
\end{proof}

\begin{lemma}\label{pbw} Assume that the equation $t^2+t+1=0$ has a solution in $\K$ and $a,s\in\K$ are such that $s\neq 0$, $(a^3,s^3)\neq(-1,-1)$ and $(1-a)^3=s^3$. Then $A=S^{a,s}$ is PBW.
\end{lemma}

\begin{proof} First, we consider the case ${\rm char}\,\K=3$. In this case the equality $(1-a)^3=s^3$ yields $s=1-a$. We shall show that the linear basis $u$, $v$, $w$ in $V$ defined by $x=u+v+w$, $y=u-v$, $z=u$ is a PBW basis in $A$. Indeed, consider $f=r_1$, $g=r_1-r_2$ and $h=r_1+r_2+r_3$, written in terms of $u$, $v$ and $w$, where $r_1=yz-azy-(1-a)xx$, $r_2=zx-axz-(1-a)yy$ and $r_3=xy-ayx-(1-a)zz$ are the defining relations of $A$. Now it is straightforward to verify that the leading monomials of $f$, $g$ and $h$ are $uv$, $uw$ and $vw$, respectively (this relies on characteristic of $\K$ being 3 and on $(a^3,s^3)\neq(-1,-1)$). By (\ref{hsaa1}), $\dim A_3=10$. By the first part of Lemma~\ref{pbbw}, $u$, $v$ and $w$ form a PBW-basis of $A$ with PBW-generators $f$, $g$ and $h$. In particular, $A$ is PBW.

From now on, we can assume that ${\rm char}\,\K\neq 3$. Let $\theta$ be a solution of the quadratic equation $t^2+t+1=0$. Then $\theta\neq 1$ and $\theta^3=1$. Since $A=S^{a,s}=Q^{-1,a,s}$ and $s^3=(1-a)^3$, Lemma~\ref{root1} allows us, without loss of generality, to assume that $s=1-a$. Then $A=Q^{-1,a,1-a}$. By Lemma~\ref{root2}, $A$ is isomorphic to $Q^{b,c,0}$, where $b=1+\theta-a$ and $c=1-a(1+\theta)$. By Lemma~\ref{dege1}, $A$ is PBW.
\end{proof}

Now we are ready to prove Theorem~\ref{copo00}. Let $(p,q,r)\in\K^3$ and $A=Q^{p,q,r}$. If $pr=qr=0$ or $p^3=q^3=r^3$, $A$ is PBW according to Lemma~\ref{dege1}. For the rest of the prove we assume that these equalities fail. That is, $r\neq 0$, $(p,q)\neq (0,0)$ and $(p^3-q^3,p^3-r^3)\neq (0,0)$. By Lemma~\ref{swap} we can without loss of generality assume that $p\neq 0$. Then $A=S^{a,s}$ with $s\neq 0$ and $(a^3,s^3)\neq (-1,-1)$, where $a=-\frac{q}{p}$ and $s=-\frac{r}{p}$. If $A$ is PBW, Lemma~\ref{nonpbw} yields that the equation $t^2+t+1=0$ is solvable in $\K$ and $s^3=(1-a)^3$. The latter equation is equivalent to $(p+q)^3+r^3=0$. Conversely, if $(p+q)^3+r^3=0$ and $t^2+t+1=0$ is solvable in $\K$, then $s^3=(1-a)^3$ and $A$ is PBW according to Lemma~\ref{pbw}. This completes the proof of Theorem~\ref{copo00}.

\section{Corollaries on Calabi--Yau property of Sklyanin algebras for various parameters \label{s50}}

As a byproduct of the exactness of the Koszul complex, we just proved, we can get the following corollary.

\begin{corollary}\label{CY} The Sklyanin algebra $Q^{p,q,r}$ is CY if and only if there are at least two non-zero parameters among $p$, $q$ and $r$ and the equation $p^3=q^3=r^3$ fails.
\end{corollary}

To explain this we need to remind few facts.

\begin{definition}\label{CY1}
An associative algebra $A$ is called $n$-{\it CY} if there exists a projective bimodule resolution ${\cal P}^{\bullet}$ of $A$ such that ${\rm Hom}({\cal P}^{\bullet},A\otimes A)\sim{\cal P}^{n-\bullet}$ or, equivalently, the derived category of $A$-bimodules satisfies Serre's duality.
\end{definition}

There is a standard way, see, for example, \cite{BW} to construct a self-dual complex $C_W$ of $A$-bimodules for algebras given by a (super)potential $W$, using the non-commutative differential.
First, for $k\leq l$, we denote by $[\cdot,\cdot]:(V^*)^{\otimes k}\times V^{\otimes l}\to V^{\otimes(l-k)}$, the bilinear map given by
$$
[\phi_1\otimes{\dots}\otimes\phi_k,\omega_1\otimes{\dots}\otimes\omega_l]=
\langle \phi_k\otimes{\dots}\otimes\phi_1,\omega_1\otimes{\dots}\otimes\omega_k\rangle \omega_{k+1}\otimes{\dots}\otimes\omega_l,
$$
where $\langle\cdot,\cdot\rangle$ is the natural pairing on $(V^*)^{\otimes k}\times V^{\otimes k}$ coming from the standard identifying of $(V^*)^{\otimes k}$ with $(V^{\otimes k})^*$.
When $A$ is potential with the potential $W\in V^{\otimes n}$ and $0\leq k\leq n$, we define
$$
\Delta^W_k:(V^*)^{\otimes k}\to V^{\otimes(n-k)},\quad \Delta^W_k(\psi)=[\psi,W].
$$
Then $W_{n-k}=\Delta^W_k((V^*)^{\otimes k})$ is a linear subspace of $V^{\otimes(n-k)}$. These spaces allow us to define the following complex $C_W$ of $A$-bimodules:
$$
0\to A\otimes W_n\otimes A\mathop{\longrightarrow}^{d_{n}}\dots {\longrightarrow}^{d_{2}}A\otimes W_1\otimes A\mathop{\longrightarrow}^{d_{1}}A\otimes W_0\otimes A\to 0,
$$
where $d_j=\epsilon_j(S_L+(-1)^j)S_R$ with $\epsilon_j=(-1)^{j(n-j)}$ if $j<\frac{n+1}2$ and $\epsilon_j=1$ otherwise, $S_L(a\otimes v_1\dots v_j\otimes b)=av_1\otimes v_2\dots v_j\otimes b$ and $S_R(a\otimes v_1\dots v_j\otimes b)=a\otimes v_1\dots v_{j-1}\otimes v_jb$.

It is proved in \cite{BW}[Lemma~6.5] that this complex is always self-dual and in the case when $A$ is quadratic, it is a subcomplex of the Koszul bimodule complex, which is the Koszul complex with the rightmost $\K$ removed tensored by $A$ on the right (this turns it into a bimodule complex). In particular, $W_j\subseteq (A_j^!)^*$ and the corresponding maps match. Moreover, it is shown in \cite{BW}[Theorem~6.2] that if $A$ is quadratic and Koszul, then $A$ is CY if and only if the complex $C_W$ coincides with the Koszul bimodule complex. The latter happens if and only if ${\rm dim}\,W_j={\rm dim}\,A_j^!$ when $j\leq n$ and $A_j^!=0$ for $j>n$. Now everything boils down to computing the dimensions of $W_j$ for Sklyanin algebras (depending on parameters).

The relations of the Sklyanin algebra $Q^{p,q,r}$ are the noncommutative partial derivatives of the
potential
$$
W=r(x^3+y^3+z^3)+p(xzy+zyx+zxy)+q(yxz+xzy+zyx).
$$
We shall from the start exclude the mega-degenerate case $p=q=r=0$. It is easy to see that for $\Delta^W_3:(V^*)^{\otimes 3}\to\K$, $\Delta^W_3(xxx)=r$, $\Delta^W_3(zyx)=p$ and $\Delta^W_3(zxy)=q$, which yields ${\rm dim}\, W_0=1$. Next, for $\Delta^W_2:V^*\otimes V^*\to V$, we have $\Delta^W_2(xx)=rx$, $\Delta^W_2(zy)=px$, $\Delta^W_2(yz)=qx$, $\Delta^W_2(yy)=ry$, $\Delta^W_2(xz)=py$, $\Delta^W_2(xz)=qy$, $\Delta^W_2(zz)=rz$, $\Delta^W_2(yx)=pz$ and $\Delta^W_2(xy)=qz$. Since $(p,q,r)\neq(0,0,0)$, the image of $\Delta^W_2$ contains the basis $x$, $y$, $z$ of $V$ and therefore ${\rm dim}\, W_1=3$. For $\Delta^W_1:V^*\to V\otimes V$, we have $\Delta^W_1(x)=rxx+pyz+qzy$, $\Delta^W_1(y)=ryy+pzx+qxz$ and  $\Delta^W_1(z)=rzz+pxy+qyx$. Since these are linearly independent ${\rm dim}\, W_2=3$. Finally, for $\Delta^W_0:\K\to V^{\otimes3}$, $\Delta^W_0(1)=W\neq 0$ and therefore ${\rm dim}\, W_3=1$.

According to Lemma~\ref{duser}, ${\rm dim}\,W_j={\rm dim}\,A_j^!$ for $j\leq 3$ and $A_j^!=0$ for $j>3$ for $A=Q^{p,q,r}$ whenever there are at least two non-zero parameters among $p$, $q$ and $r$ and the equation $p^3=q^3=r^3$ fails. Under these assumptions, the Koszul bimodule complex provides a self-dual resolution, which ensures the CY property. In the remaining cases, the equalities ${\rm dim}\,W_j={\rm dim}\,A_j^!$ break, since according to Lemma~\ref{duser} $H_{A^!}(t)=\textstyle\frac{1+2t}{1-t}$. For instance, ${\rm dim}\,A^!_3=3\neq 1={\rm dim}\,W_3$. Hence $A$ is not CY in these degenerate cases.

This type of argument provides a way to check the CY property using $H_{A^!}$. If one has a Koszul potential quadratic algebra, then the CY property is equivalent to the equalities ${\rm dim}\,W_j={\rm dim}\,A_j^!$.

\section{Generalized Sklyanin algebras \label{s5}}

Let $\xi=(p_1,p_2,p_3,q_1,q_2,q_3,r_1,r_2,r_3)\in\K^9$. In this section we consider the $\K$-algebras $\widehat{Q}^\xi$ given by the generators $x$, $y$ and $z$ and the relations

\begin{equation}\label{skl-gen}
\text{$p_1yz+q_1zy+r_1xx=0$},\quad \text{$p_2zx+q_2xz+r_2yy=0$},\quad\text{$p_3xy+q_3yx+r_3zz=0$}.
\end{equation}
We call these {\it generalized Sklyanin algebras}. The actual Sklyanin algebras correspond to the case $p_1=p_2=p_3$, $q_1=q_2=q_3$ and $r_1=r_2=r_3$. We will demonstrate that $3$-parameter Sklyanin algebras $Q^{p,q,r}$, coming from nature, are very different and specific, comparing to other their relatives from the class of generalized Sklyanin algebras. Indeed,
seemingly innocuous generalization leads to a dramatic changes in the behavior.

We know that generic Sklyanin algebras are Koszul PHSs. This is no longer the case for generalized Sklyanin algebras.

\begin{theorem}\label{copo02} For $\xi$ from a non-empty Zarisski open subset of $\K^9$, both $A=\widehat{Q}^\xi$ and $A^!$ are finite dimensional.
\end{theorem}

Note that when both $A=\widehat{Q}^\xi$ and $A^!$ are finite dimensional, their Hilbert series are non-constant polynomials and therefore (\ref{stm2}) fails. Thus $A$ is non-Koszul. Hence Theorem~\ref{copo02} yields that if $\K$ is infinite, a Zarisski-generic $\widehat{Q}^\xi$ is finite dimensional and non-Koszul. We can actually determine the minimal Hilbert series of a generalized Sklyanin algebra:
$$
H_{\min}(t)=\sum_{n=0}^\infty d_nt^n,\ \ \text{where}\ \ d_n=\min_{\xi\in\K^9} \dim \widehat{Q}^\xi_n.
$$

\begin{theorem}\label{copo22} If $\K\neq\Z_2$ $(\K$ is not the $2$-element field$)$, then the minimal Hilbert
series $H_{\min}$ of a generalized Sklyanin algebra is given by $H_{\min}(t)=1+3t+6t^2+9t^3+9t^4$. If $\K=\Z_2$, then $H_{\min}(t)=1+3t+6t^2+9t^3+9t^4+5t^5+t^6$. In any case, there exists a generalized Sklyanin algebra $A$ such that $H_A=H_{\min}$.
\end{theorem}

\begin{remark}\label{rem2} By Theorem~\ref{copo22} and Lemma~\ref{minhs}, if $\K$ is an infinite field, a Zarissky-generic $\widehat{Q}^\xi$ has the Hilbert series $1+3t+6t^2+9t^3+9t^4$ and therefore has dimension $28$. If $\K\neq\Z_2$, the minimal dimension of $\widehat{Q}^\xi$ is again $28$, while for $\K=\Z_2$, the minimal dimension of $\widehat{Q}^\xi$ is $34$.
\end{remark}

It is possible to characterize PHS among the generalized Sklyanin algebras. The subtle bit is that the set of leading monomials of the relations depends on the distribution of zeros among the coefficients. Fortunately many cases are equivalent to each other by means of applying a permutation of variables (any permutation of variables keeps the shape of the relations and shuffles the coefficients) and scaling the variables (a substitution which multiplies each variable by a non-zero constant).

First, we describe the following 4 classes of generalized Sklyanin algebras. Namely, we say that for a generalized Sklyanin algebra $A$,
\begin{align}\label{PP1}
&\text{$A\in {\cal P}_1$ if}\qquad \left\{\begin{array}{l}\text{$A$ is a Sklyanin algebra $A=Q^{p,q,r}$ with}\\
\text{$(pq,pr,qr)\neq(0,0,0)$ and $(p^3-q^3,p^3-r^3)\neq (0,0)$;}\end{array}\right.
\\
&\text{$A\in {\cal P}_2$ if}\qquad \left\{\begin{array}{l}\text{$A$ is a generalized Sklyanin algebra, whose relations have the shape}\\ \text{$yz-azy=0$, $bzx-xz=0$, $xy-cyx=0$, where $a,b,c\in\K$;}\end{array}\right.\label{PP2}
\\
&\text{$A\in {\cal P}_3$ if}\qquad \left\{\begin{array}{l}\text{$A$ is a generalized Sklyanin algebra, whose relations have the shape}\\ \text{$yz-azy=0$, $bzx-xz+yy=0$, $xy-ayx=0$, where $a,b\in\K$;}
\end{array}\right.\label{PP3}
\\
&\text{$A\in {\cal P}_4$ if}\qquad \left\{\begin{array}{l}\text{$A$ is a generalized Sklyanin algebra, whose relations have the shape}\\ \text{$yz-azy=0$, $azx-xz+yy=0$, $xy-ayx-zz=0$, where $a\in\K$;}
\end{array}\right.\label{PP4}
\\
&\text{$A\in {\cal P}_5$ if}\qquad \left\{\begin{array}{l}\text{$A$ is a generalized Sklyanin algebra, whose relations have the shape}
\\ \text{$yz+\theta zy+\theta^2xx=0$, $zx+\theta^4xz+\theta^2 yy=0$, $xy+\theta^7yx+\theta^2 zz=0$,}\\
\text{where $\theta\in\K$ satisfies $\theta^9=1$ and $\theta^3\neq 1$;}\end{array}\right.\label{PP5}
\\
&\text{$A\in {\cal P}_6$ if}\qquad \left\{\begin{array}{l}\text{$A$ is a generalized Sklyanin algebra, whose relations have the shape}
\\ \text{$xx=0$, $\alpha zx+xz+yy=0$, $xy+\frac{1}{\alpha} yx+zz=0$, where $\alpha\in\K^*$.}\end{array}\right.\label{PP6}
\end{align}

Note that while ${\cal P}_j$ for $1\leq j\leq 4$ and ${\cal P}_6$ are infinite if $\K$ is infinite, while ${\cal P}_5$ is finite. More specifically, ${\cal P}_5$ is empty if $\K^*$ has no elements of order $9$ and contains $6$ sets of relations otherwise. Furthermore these $6$ algebras are one and the same since the permutations of the variables act transitively on the $6$-element set of algebras defined in ${\cal P}_5$.

\begin{theorem}\label{COCOCO} Assume that $\K$ is algebraically closed and let $A$ be a generalized Sklyanin algebra. Then $A$ satisfies $H_A(t)=(1-t)^{-3}$ if and only if the defining relations of $A$ can be turned into that of an algebra from ${\cal P}_j$ for some $1\leq j\leq 6$ by means of a permutation of the variables, a scaling of the variables and a normalization of the relations $($multiplying each relation by a non-zero constant$)$. Furthermore, $A$ is Koszul if $j\leq 5$ and $A$ is non-Koszul if $j=6$.
\end{theorem}

In other words, Theorem~\ref{COCOCO} says algebras in ${\cal P}_j$ with $j\leq 5$ are Koszul PHSs, algebras in ${\cal P}_6$ are PHS but non-Koszul, while the classes ${\cal P}_j$ for $1\leq j\leq 6$ cover all generalized Sklyanin PHSs up to a permutation and scaling of the variables.

\subsection{Proof of Theorem~\ref{COCOCO}}

\begin{lemma}\label{scale1} Let $\K$ be algebraically closed,
$\xi=(p_1,p_2,p_3,q_1,q_2,q_3,r_1,r_2,r_3)\in\K^9$ and
$\xi'=(p_1,p_2,p_3,q_1,q_2,q_3,r'_1,r'_2,r'_3)\in\K^9$ be such that
$r_1r_2r_3=r'_1r'_2r'_3$ and for each $j\in\{1,2,3\}$, either $r_j=r'_j=0$ or $r_jr'_j\neq0$.
Then there is a scaling of the variables providing an isomorphism between
$\widehat{Q}^\xi$ and $\widehat{Q}^{\xi'}$.
\end{lemma}

\begin{proof} For $\alpha,\beta,\gamma\in\K^*$, under the scaling substitution $x=\alpha u$, $y=\beta v$,
$z=\gamma w$, the defining relations of $\widehat{Q}^\xi$ (in terms of $u$, $v$ and $w$ after a suitable
normalization) take form
$$
\textstyle p_1vw+q_1wv+\frac{r_1\alpha^2}{\beta\gamma}uu=0,\ \
\textstyle p_2wu+q_2uw+\frac{r_2\beta^2}{\alpha\gamma}vv=0,\ \
\textstyle p_3uv+q_3vu+\frac{r_3\gamma^2}{\alpha\beta}ww=0.
$$
Thus in order to prove that a scaling providing an isomorphism between
$\widehat{Q}^\xi$ and $\widehat{Q}^{\xi'}$, it suffices to show that
\begin{equation}\label{albega}
\frac{r_1\alpha^2}{\beta\gamma}=r_1',\ \ \frac{r_2\beta^2}{\alpha\gamma}=r_2'
\ \ \text{and}\ \ \frac{r_3\gamma^2}{\alpha\beta}=r_3'\ \ \text{for some }\alpha,\beta,\gamma\in\K^*.
\end{equation}
First, assume that at least two of $r_j$ are non-zero. Without loss of generality, $r_2r_3\neq 0$.
Then $r'_2r'_3\neq 0$. Since $\K$ is algebraically closed, there is $\beta\in\K^*$ such that
$\beta^3=\frac{r'_3{r'_2}^2}{r_3r_2^2}$. Now we choose $\alpha=1$ and $\gamma=\frac{r_2\beta^2}{r'_2}$.
It is routine to verify that (\ref{albega}) is satisfied.

The case $r_1=r_2=r_3=0$ is trivial. It remains to consider the case when exactly one of $r_j$ is zero.
Without loss of generality, $r_1=r_2=0$ and $r_3\neq 0$. Then $r'_1=r'_2=0$ and $r'_3\neq 0$. Now choosing
$\beta=\gamma=1$ and $\alpha=\frac{r_3}{r'_3}$, we see that (\ref{albega}) is satisfied.
\end{proof}

\begin{lemma}\label{kopsa} Let $1\leq j\leq 5$ and $A\in {\cal P}_j$. Then $A$ is a Koszul PHS.
\end{lemma}

\begin{proof} By Remark~\ref{rem1}, we can without loss of generality assume that $\K$ is algebraically closed.
The case $A\in{\cal P}_1$ follows from Theorem~\ref{copo0}. In the case $A\in{\cal P}_j$ with $2\leq j\leq 4$,
it is routine to verify that the defining relations of $A$ form a Gr\"obner basis of the
ideal they generate. Thus $A$ is PBW and therefore Koszul. By Proposition~\ref{pbbw}, $A$ is a PHS.
It remains to consider the case $A\in {\cal P}_5$. Let $\theta\in\K$ be such that $\theta^9=1\neq\theta^3$.
Then $\theta^6+\theta^3+1=0$. This equality yields
$(1+\theta)(1+\theta^4)(1+\theta^7)=-\theta^6$. By Lemma~\ref{scale1}, $A$ is isomorphic
to the algebra given by the generators $x$, $y$ and $z$ and the relations $g_1=g_2=g_3=0$, where
$$
g_1=yz+\theta^7 zy-(1+\theta^7)xx, \ \ g_2=zx+\theta^4 xz-(1+\theta^4)yy, \ \
g_3=xy+\theta yx-(1+\theta)zz.
$$
Now we perform the linear substitution $x=(1-\theta^7)u+\theta^2v+w$, $y=(\theta^3-\theta^7)u+\theta^5v+\theta^6w$ and $z=(\theta^6-\theta^7)u+\theta^8v+\theta^3w$. With respect to the new basis $u,v,w$ in $V$, the space $R$ of quadratic relations of $A$ is spanned by $uu-\theta^4uw+\theta^2vu+\theta^7wu-\theta^3wv$, $uv+\theta^2uw-vu-\theta^5wu+(\theta-\theta^4)wv$ and $vw-\theta^6wv$. We proceed to compute the reduced Gr\"obner basis of the ideal of relations with respect to the degree-lexicographical ordering assuming $u>v>w$. The basis happens to be finite and comprises the defining relations together with four more elements with the leading monomials $uwu$, $uwv$, $uwwv$ and $uwww$. The full list of the leading monomials is $uu$, $uv$, $vw$, $uwu$, $uwv$, $uwwv$ and $uwww$. Now one can see that the normal words are $w^jv^k(uww)^m$, $w^jv^k(uww)^mu$ and $w^jv^k(uww)^muw$ with $j,k,m\in\Z_+$. The number of such words of length $n$ is exactly $\frac{(n+1)(n+2)}2$ and therefore $H_A(t)=(1-t)^{-3}$ and $A$ is a PHS. Since the set of normal words is closed under multiplication by $w$ on the left, the said multiplication is an injective linear map on $A$ and $A$ has no non-trivial right annihilators. Now an application of Corollary~\ref{deg3a} yields Koszulity of $A$.
\end{proof}

\begin{lemma} \label{pp6} Let $A\in{\cal P}_6$. Then $A$ is a non-Koszul PHS.
\end{lemma}

\begin{proof} By swapping $x$ and $z$, we see that $A$ is isomorphic to the algebra $B$ given $zz=0$, $xx=-\alpha^{-1}yz-zy$ and $yy=-\alpha xz-zx$ with $\alpha\in\K^*$. One can show that for a generic $\alpha$, the Gr\"obner basis of the ideal of relations of this algebra is infinite with respect to the deg-lex ordering. However, it is finite with respect to another ordering compatible with multiplication. Namely, we use the following order on the monomials. A monomial of bigger degree is bigger. For two monomials of equal degree, the one with higher $z$-degree is smaller. For two monomials of equal degrees and $z$-degrees, we replace all $y$ by $x$ in both and compare the resulting $xz$ monomials using left-to-right deg-lex order assuming $x>z$. If this happens to be a tie, we break it by using left-to-right deg-lex order assuming $x>y>z$. In this order the Gr\"obner basis consists of the defining relations together with $xyz+\alpha xzy-yzx-\alpha zyx$, $\alpha yxz-\alpha xzy+yzx-zxy$ and $\alpha xzyz-yzxz-\alpha zxzy+zyzx$. The leading monomials of the basis are $xx$, $yy$, $zz$, $xyz$, $xzy$ and $xzyz$, which leads to $H_A(t)=(1-t)^{-3}$. That is, $A$ is a PHS.

The dual algebra $A^!$ is given by the relations $xy=yx=yy-zx=xx-zy=\alpha yy-xz=xx-\alpha yz=0$. Clearly, $A^!$ is infinite dimensional since $z^n$ are linearly independent in $A^!$ (in fact, a Gr\"obner basis computation  gives $H_{A^!}(t)=1+3t+3t^2+t^3+t^4+t^4+{\dots}$). Hence the equality $H_A(-t)H_{A^!}(t)=1$ fails and $A$ is non-Koszul.
\end{proof}

In order to complete the proof of Theorem~\ref{COCOCO} it remains to show that if $\K$ is algebraically closed and let $A$ generalized Sklyanin PHS, then $A$ falls into one of the families ${\cal P}_j$ for $1\leq j\leq 6$ after suitable permutation and scaling of variables (together with normalization of relations, of course). The consideration splits into cases according to how zeros are distributed among the coefficients. We can assume from the start that none of the defining relations of $A$ vanishes. Indeed, otherwise $\dim A_2>6$ and $A$ is not a PHS. The six cases $p_jq_kr_l\neq0$ for $\{j,k,l\}=\{1,2,3\}$ can be obtained from one another by suitable permutations of variables. If $p_jq_kr_l=0$ for every $j,k,l$ satisfying $\{j,k,l\}=\{1,2,3\}$, then the matrix
$$
\left(\begin{array}{ccc}p_1&q_1&r_1\\ p_2&q_2&r_2\\ p_3&q_3&r_3\end{array}\right)
$$
has either a zero column or a zero $2\times 2$ submatrix (the case of a zero row is excluded by the assumption that none of the relations is zero). Thus up to a permutation of the variables, we have only to deal with the cases
\begin{itemize}\itemsep=-2pt
\item$p_3q_2r_1\neq 0$;
\item$q_1=q_2=q_3=0$;
\item$r_1=r_2=r_3=0$;
\item$q_1=q_3=r_1=r_3=0$;
\item$p_1=p_2=q_1=q_2=0$.
\end{itemize}

First, we deal with easier cases. If $p_1=p_2=q_1=q_2=0$ is satisfied, the relations of $A$ (up to a normalization) take shape $xx=0$, $yy=0$ and $p_3xy+q_3yx+r_3zz=0$. Regardless which monomial is leading in the last relation, computing the degree $3$ elements of the Gr\"obner basis, we easily see that $\dim A_3\geq 11$ (it is actually either $11$ or $12$). Hence $\dim A_3\neq 10$ and $A$ is not a PHS.

If $q_1=q_3=r_1=r_3=0$ is satisfied, the relations of $A$ (up to a normalization) take shape $yz=0$, $xy=0$ and $p_2zx+q_2xz+r_2yy=0$. Regardless which monomial is leading in the last relation, computing the degree $3$ elements of the Gr\"obner basis, we again see that $\dim A_3\geq 11$ (it is $11$, $12$ or $13$). Hence $\dim A_3\neq 10$ and $A$ is not a PHS.

If $r_1=r_2=r_3=0$ is satisfied, then either $A$ belongs to ${\cal P}_2$ or $A$ is a monomial algebra satisfying $\dim A_3=12$. In the latter case $A$ is not a PHS.

The case $q_1=q_2=q_3=0$ is slightly more involved. If at least two of $r_j$ equal $0$, we can without loss of generality assume that $r_1=r_2=0$. The relations of $A$ take the shape $yz=0$, $zx=0$ and $p_3xy+r_3zz=0$. Again, it is easy to see that $\dim A_3\geq 11$ and therefore $A$ is not a PHS. It remains to consider the case when at least two of $r_j$ are non-zero. Without loss of generality $r_1r_2\neq 0$. First, consider the case $p_1=0$. Then the relations take shape $xx=\alpha yz$, $yy=\beta zx$ and $zz=0$ with $\alpha=-\frac{p_1}{r_1}$ and $\beta=-\frac{p_2}{r_2}$. If $\alpha\beta=0$, then we have $\dim A_3>10$ and $A$ is not a PHS. If $\alpha\beta\neq 0$, Lemma~\ref{scale1} allows us by means of a scaling of variables to bring the relations to $xx=yz$, $yy=zx$ and $zz=0$. It is a tedious enough but a doable exercise to check that $\dim A_6=31\neq 28$ and this implies that $A$ is not a PHS (actually $6$ is the first degree for which $\dim A_j$ deviates from $\frac{(j+1)(j+2)}{2}$). It remains to consider the case $r_1r_2p_3\neq 0$. Then the relations of $A$ take shape $xx=\alpha yz$, $yy=\beta zx$ and $xy=\gamma zz$ with $\alpha=-\frac{p_1}{r_1}$ $\beta=-\frac{p_2}{r_2}$ and $\gamma=-\frac{r_3}{p_3}$. If at leat two of the numbers $\alpha$, $\beta$ and $\gamma$ are $0$, $\dim A_3\geq 11$ and $A$ is not a PHS. If $\alpha=0$, $\beta\gamma\neq 0$ or $\beta=0$, $\alpha\gamma\neq 0$, then by a permutation and scaling of the variables (using Lemma~\ref{scale1}), we get the familiar relations $xx=yz$, $yy=zx$ and $zz=0$. We already know that then $\dim A_6=31$ and therefore $A$ is not a PHS. If $\gamma=0$, $\alpha\beta\neq 0$, by Lemma~\ref{scale1}, a scaling of the variables turns the relations into $xx=yz$, $xy=0$, $yy=zx$. In this case, using the Gr\"obner basis technique, one easily checks that $\dim A_4=17\neq 15$ and therefore $A$ is not a PHS. Finally, if $\alpha\beta\gamma\neq 0$, using the fact that $\K$ is algebraically closed, we can find $r\in\K^*$ such that $r^3=\frac{\gamma}{\alpha\beta}$. Now Lemma~\ref{scale1} provides a scaling of the variables, which turns the relations into $yz-rxx=0$, $zx-ryy=0$ and $xy-rzz=0$. These are the relations of $Q^{1,0,-r}\in{\cal P}_1$.

It remains to consider the main (and most involved) case $p_3q_2r_1\neq 0$. We treat it in more detail. In this case we can write the relations of $A$ as $xx=ayz+\alpha zy$, $xy=byx+\beta zz$ and $xz=cyy+\gamma zx$, where $a=-\frac{p_1}{r_1}$, $b=-\frac{q_3}{p_3}$, $c=-\frac{r_2}{q_2}$, $\alpha=-\frac{q_1}{r_1}$, $\beta=-\frac{r_3}{p_3}$, $\gamma=-\frac{p_2}{q_2}$. The leading monomials $xx$, $xy$ and $xz$ of the relations admit 3 overlaps $xxx$, $xxy$ and $xxz$. Resolving these, we find that the degree $3$ part of the Gr\"obner basis of the ideal of the relations comprises of
\begin{align*}
\xi_1&=c(ab+\alpha)yyy+a(b\gamma-1)yzx+\alpha(b\gamma-1)zyx+\beta(\alpha\gamma+a)zzz,
\\
\xi_2&=(ab^2+c\beta)yyz+(\alpha b^2-a)yzy+(c\beta\gamma-\alpha)zyy+\beta(b+\gamma^2)zzx,
\\
\xi_3&=c(b^2+\gamma)yyx+(bc\beta-a)yzz+(a\gamma^2-\alpha)zyz+(\alpha\gamma^2+c\beta)zzy.
\end{align*}
Since there are exactly 12 degree $3$ monomials, which do not contain any of $xx$, $xy$ and $xz$ as a submonomial, $\dim A_3=12-d$, where $d$ is the dimension of the space spanned by $\xi_1$, $\xi_2$ and $\xi_3$. Thus $A$ can not be a PHS unless $d=2$. Since no monomial features in more than one of $\xi_j$, $d$ equals $2$ precisely when exactly one of $\xi_j$ equals $0$. Now, solving the corresponding systems of algebraic equations, we see that
\begin{align*}
\xi_1&=0\iff \text{$b\gamma{-}1{=}ab{+}\alpha{=}0$ OR $b\gamma{-}1{=}\beta{=}c{=}0$ OR $a{=}\alpha{=}0$},
\\
\xi_2&=0\iff \text{$a{=}\alpha{=}\beta{=}0$ OR $b{+}\gamma^2{=}\alpha{=}a{=}c{=}0$ OR
$\gamma^9{+}1{=}b{+}\gamma^2{=}\alpha{-}c\beta\gamma{=}a{-}c\beta\gamma^5{=}0$},
\\
\xi_3&=0\iff \text{$a{=}\alpha{=}c{=}0$ OR $b^2{+}\gamma{=}\alpha{=}a{=}\beta{=}0$ OR
$b^9{+}1{=}b^2{+}\gamma{=}a{-}c\beta b{=}\alpha{-}c\beta b^5{=}0$}.
\end{align*}

Using the above display, it is easy to see that
\begin{align*}
&\xi_2=0,\ \xi_1\neq0,\ \xi_3\neq 0\iff \text{$\gamma^9{+}1{=}b{+}\gamma^2{=}\alpha{-}c\beta\gamma{=}a{-}c\beta\gamma^5{=}0$ and
$c\beta(\gamma^3{+}1)\neq 0$}.
\\
&\xi_3=0,\ \xi_1\neq0,\ \xi_2\neq 0\iff \text{$b^9{+}1{=}b^2{+}\gamma{=}a{-}c\beta b{=}\alpha{-}c\beta b^5{=}0$ and
$c\beta(b^3{+}1)\neq 0$}.
\\
&\xi_1=0,\ \xi_2\neq0,\ \xi_3\neq 0\iff \text{$a{=}\alpha{=}0{\neq}c\beta$ OR $b\gamma{-}1{=}\beta{=}c{=}0{\neq}a\alpha$}
\\
&\qquad\qquad\qquad\qquad\qquad\qquad\,\,\,\,\text{OR $b\gamma{-}1{=}ab{+}\alpha{=}0{\neq}a\alpha$ and $(b^3{+}1,c\beta{-}a\gamma){\neq}(0,0)$.}
\end{align*}
Since $\dim A_3=10$ precisely when exactly one of $\xi_j$ is $0$, we can restrict ourselves to this case. If only $\xi_2$ vanishes or only $\xi_3$ vanishes, the above display yields that there is $\theta\in\K$ such that $\theta^9=1\neq\theta^3$ and after a permutation of the variables, the relations of $A$ take shape $yz+\theta zy+s_1xx$, $zx+\theta^4 xz+s_2yy$ and $xy+\theta^7 yx+s_3zz$ with $s_1s_2s_3=\theta^6$. Now Lemma~\ref{scale1} implies that a scaling of the variables brings the relations to that of an algebra in ${\cal P}_6$. It remains to consider the case $\xi_1=0$, $\xi_2\neq0$ and $\xi_3\neq 0$. By the above display, $a=\alpha=0\neq c\beta$ OR $b\gamma-1=\beta=c=0\neq a\alpha$ OR $b\gamma-1=ab+\alpha=0\neq a\alpha$ and $(b^3+1,c\beta-a\gamma)\neq(0,0)$. In the case $b\gamma-1=\beta=c=0\neq a\alpha$, Lemma~\ref{scale1} provides a scaling of the variables bringing the relations to that of an algebra from ${\cal P}_3$. Now assume that $b\gamma-1=ab+\alpha=0\neq a\alpha$ and $(b^3+1,c\beta-a\gamma)\neq(0,0)$. If $\beta=c=0$, we fall into the previous case. Thus we can assume that $(\beta,c)\neq(0,0)$. The equality $b\gamma-1=ab+\alpha=0\neq a\alpha$ yields that after a normalization the relations take shape $yz-bzy+s_1xx$, $zx-bxz+s_2yy$ and $xy-byx+s_3zz$, where $s_1=-\frac1a$, $s_2=\frac{c}{\gamma}$, $s_3=-\beta$. Moreover, at least two of $s_j$ are non-zero. If there is $j$ with $s_j=0$, then after a permutation and a scaling of the variables (use Lemma~\ref{scale1}), we bring the relation to ${\cal P}_4$. If $s_1s_2s_3\neq 0$, we use algebraic closeness of $\K$ to find $t\in\K^*$ such that $t^3=s_1s_2s_3=\frac{c\beta}{a\gamma}$. By Lemma~\ref{scale1}, we can turn the relations into $yz-bzy+txx$, $zx-bxz+tyy$ and $xy-byx+tzz$, which are the
relations of $Q^{1,-b,t}$. Since $(b^3+1,c\beta-a\gamma)\neq(0,0)$, the equality $1=-b^3=t^3$ fails. Since $bt\neq0$, we have fallen into the class ${\cal P}_1$.

It remains to consider the case $a=\alpha=0\neq c\beta$. In this case after a scaling provided by Lemma~\ref{scale1}, the defining relations of $A$ take form $xx=0$, $xy-byx+zz=0$ and $xz+yy-\gamma zx=0$. Computing the Gr\"obner basis up to degree $4$, we get $\dim A_4=14\neq 15$ (and therefore $A$ is not a PHS) unless $\gamma b=1$. On the other hand, if $\gamma b=1$, these relations fall into ${\cal P}_6$. This concludes the proof of
Theorem~\ref{COCOCO}.

\subsection{Proof of Theorems~\ref{copo02} and~\ref{copo22}}

\begin{lemma}\label{PHEW0} Assume that ${\rm char}\,\K\in\{3,5\}$. Then the generalized Sklyanin algebra $A$ given by the relations $xx+zy=0$, $xy+2yx+zz=0$ and $xz+zx+yy=0$ satisfies $H_A(t)=1+3t+6t^2+9t^3+9t^4$ and $H_{A^!}(t)=1+3t+3t^2$.
\end{lemma}

\begin{proof} One easily sees that $A^!$ is given by the relations $yz=0$, $xx=zy$, $xy=zz$, $yx=2zz$, $xz=zx$ and $yy=zx$.
The ideals of relations of $A$ and of $A^!$ happen to have a finite Gr\"obner bases. In the case ${\rm char}\,\K=3$, the elements $xx+zy$, $xy+2yx+zz$, $xz+zx+yy$, $yz^2+zyz$, $y^2z-yzy+z^2x$, $y^3-zyx-z^3$, $yzyz-z^2y^2-z^3x$, $yzy^2-zyzx+z^4$, $yzyx-zy^2x-z^3y$, $z^5$, $z^4y$, $z^4x$, $z^3yz$, $z^3y^2$, $z^3yx$, $z^2yzy$, $z^2yzx$ and $z^2y^2x$ form a Gr\"obner basis of the ideal of relations of $A$. The complete list of corresponding normal words is $x$, $y$, $z$, $yx$, $yy$, $yz$, $zx$, $zy$, $zz$, $yyx$, $yzx$, $zyx$, $zzx$, $yzy$, $zyy$, $zzy$, $zyz$, $zzz$, $zyyx$, $zzyx$, $zyzx$, $zzzx$, $zzyy$, $zyzy$, $zzzy$, $zzyz$ and $zzzz$.

In the case ${\rm char}\,\K=5$, the elements $xx+zy$, $xy+2yx+zz$, $xz+zx+yy$, $y^2z+yzy+z^2x$, $y^3-zyx-z^3$, $y^2x-yz^2-2zyz$,  $zyzx+2z^2yx$, $yz^3+zyz^2-2z^3y$, $yz^2y+zyzy+2z^2y^2-2z^3x$, $yzyz+zyzy+z^2y^2$, $yzy^2-yz^2x-3z^2yx-z^4$, $yzyx-3zyz^2-z^2yz+z^3y$, $z^5$, $z^4y$, $z^4x$, $z^3yz$, $z^3y^2$, $z^3yx$, $z^2yz^2$, $z^2yzy$ and $zyz^2x$ form a Gr\"obner basis of the ideal of relations of $A$. The complete list of corresponding normal words is $x$, $y$, $z$, $yx$, $yy$, $yz$, $zx$, $zy$, $zz$, $yzx$, $zyx$, $zzx$, $yzy$, $yzz$, $zyy$, $zzy$, $zyz$, $zzz$, $yzzx$, $zzyx$,  $zzzx$,  $zzyy$, $zyzy$, $zzzy$, $zzyz$, $zyzz$ and $zzzz$.

In both cases the elements $yz$, $xx-zy$, $xy-zz$, $yx-2zz$, $xz-zx$, $yy-zx$, $zzx$, $zzy$ and $zzz$ form a Gr\"obner basis of the ideal of relations of $A^!$. The complete list of corresponding normal words is $x$, $y$, $z$, $zx$, $zy$ and $zz$. In any case, we have $H_A(t)=1+3t+6t^2+9t^3+9t^4$ and $H_{A^!}(t)=1+3t+3t^2$.
\end{proof}

It is worth mentioning that in the above lemma, the condition ${\rm char}\,\K\in\{3,5\}$ can be significantly relaxed. For instance, the same conclusion holds if ${\rm char}\,\K\in\{0,11,13,17\}$ as well as for any sufficiently large prime characteristic. On the other hand, the conclusion of Lemma~\ref{PHEW0} fails if ${\rm char}\,\K\in\{2,7,19,23\}$.

\begin{lemma}\label{PHEW} Let $a\in\K$ be such that
\begin{equation}\label{aneq}
%\begin{array}{l}
\kern-15pt0{\neq}a(1{-}a)(3{+}a^2)(2{-}2a{-}a^3)(1{-}3a{-}a^3)(1{+}a{-}3a^2\!{-}a^3\!{-}a^4)(1{-}2a{+}3a^2\!{-}a^3\!{+}a^4)
(1{-}2a^2\!{+}3a^3\!{+}a^5).
%\end{array}
\end{equation}
Then the generalized Sklyanin algebra $A$ given by the relations $xx=zy$, $xy=zz$ and $xz=yy+azx$ satisfies $H_A(t)=1+3t+6t^2+9t^3+9t^4$ and $H_{A^!}(t)=1+3t+3t^2$.
\end{lemma}

\begin{proof} One easily sees that $A^!$ is given by the relations $xx=-zy$, $xy=-zz$, $xz=-\frac1a zx$, $yx=0$, $yy=\frac1a zx$, $yz=0$. A direct computation yields that these defining relations together with $zzx$, $zzy$ and $zzz$ provide a Gr\"obner basis in the ideal of relations of $A^!$. The only normal words are $x$, $y$, $z$, $zx$, $zy$ and $zz$, which gives $H_{A^!}(t)=1+3t+3t^2$. Note that for this to be satisfied we just need $a\neq 0$. The rest of the assumptions on $a$ are needed to deal with $H_A$, which we start right now.

Resolving all three overlaps $(xx)z=x(xz)$, $(xx)x=x(xx)$ and $(xx)y=x(xy)$ of the leading monomials $xx$, $xy$ and $xz$ of the defining relations of $A$, we see that (provided $a\neq 0$), the degree $3$ part of the Gr\"obner basis of the ideal of relations of $A$ consists of the following $3$ elements
$$
\textstyle yyx-\frac1a zyz+\frac{1+a^2}{a}zzy,\ \ yyy-zyx+azzz,\ \ yyz-(1-a)zyy+a^2zz.
$$
It follows that the complete list of degree $3$ normal words is $yzx$, $yzy$, $yzz$, $zyx$, $zyy$, $zyz$, $zzx$, $zzy$ and $zzz$. Since there are $9$ of them, $\dim A_3=9$. There are nine overlaps of degree $4$: $xyyx$, $xyyy$, $xyyz$, $yyxx$, $yyxy$, $yyxz$, $yyyx$, $yyyy$ and $yyyz$. Resolving them, we find that (provided $a\notin\{0,1\}$), the degree $4$ part of the Gr\"obner basis of the ideal of relations of $A$ consists of the following $6$ elements:
$$
\begin{array}{ll}
\textstyle yzyx-ayzzz-\frac{1+a}{1-a}zzyz+\frac{2a-a^2+a^3}{1-a}zzzy,&zyzx-(1+a)zzyx+a^2zzzz,\\
\textstyle yzyy-\frac{a^2}{1-a}yzzx-\frac{1+a+a^2}{1-a}zzyx+\frac{2a+a^3}{1-a}zzzz,&zyzy-(1+a-a^2)zzyy+a^3zzzx,\\
yzyz-(1+a^2)yzzy-(a+a^2-a^3)zzyy+(a^2+a^4)zzzx,&\textstyle zyzz-\frac{1+a}{1-a}zzyz+\frac{a+a^3}{1-a}zzzy.
\end{array}
$$
It follows that the complete list of degree $4$ normal words is $yzzx$, $yzzy$, $yzzz$, $zzyx$, $zzyy$, $zzyz$,
$zzzx$, $zzzy$ and $zzzz$. Since there are $9$ of them, $\dim A_4=9$. It also follows that
\begin{equation}\label{AA5}
\text{$A_5$ is spanned by} \ \ z^3y^2,\ z^4x,\ z^3yz,\ z^4y,\ z^3yx,\ z^5,\ yzzyx,\ yzzyz,\ yzzyy,\ yz^4,\ yz^3y,\ yz^3x.
\end{equation}
The monomials listed in the above display are all degree $5$  monomials which do not contain a smaller degree subword being the leading monomial of a member of the Gr\"obner basis of the ideal of relations of $A$.

Instead of resolving all degree $5$ overlaps, which is tedious indeed, we just resolve enough of them to show that $A_5=\{0\}$. We start by resolving and reducing the overlaps $zyzxx$, $zyzyz$, $zyzxy$, $zyzyx$, $zyzxz$ and $zyzyy$, which provide (respectively) the following equalities in $A$:
\begin{equation}\label{gbup51}
\begin{array}{l}
(2-a-a^2)z^3y^2-(1-a-a^2-a^3)z^4x=0,\\ (3+3a-5a^2+a^3-2a^4)z^3y^2-(3a^3+a^5)z^4x=0, \\
(a^2+3)(z^3yz-az^4y)=0,\\ (1-2a-3a^2+3a^3-3a^4+a^5-a^6)z^3yz-(1-a-3a^2+4a^3-5a^4+a^5-a^6)z^4y=0, \\
(1-3a)z^3yx-(1-a-a^2-a^3)z^5,\\ (1+3a-3a^2+2a^3-a^4)z^3yx-(1-a+2a^2-a^3+a^4)z^5=0.
\end{array}
\end{equation}
The determinant of the $2\times 2$ matrix of the coefficients of the equations in the first two lines of the above display is $(1-a)(3+a^2)(1+a-3a^2-a^3-a^4)$. By (\ref{aneq}) it is non-zero and therefore the first line of the above display yields that $z^3y^2=z^4x=0$ in $A$. For the third and fourth lines, the determinant is $(1-a)^2(3+a^2)(1-2a^2+3a^3+a^5)$ and it does not vanish by (\ref{aneq}). Thus $z^3yz=az^4y=0$ in $A$. For last two lines, the determinant is $a(1-a)(3+a^2)(2-2a-a^3)$. Again, it does not vanish by (\ref{aneq}). Thus $z^3yx=z^5=0$ in $A$. Summarizing, we get
\begin{equation}\label{gbup51a}
z^3y^2=z^4x=z^3yz=z^4y=z^3yx=z^5=0\ \ \text{in $A$}.
\end{equation}
Now we resolve and reduce (using the degrees $\leq 4$ part of the Gr\"obner basis together with (\ref{gbup51a})) the overlaps $yzyzx$, $yzyxz$, $yzyxy$, $yzyyx$, $yzyzy$ and $yzyxx$, which provide (respectively) the following equalities in $A$:
\begin{equation}\label{gbup52}
\begin{array}{ll}
(1-a)yzzyx-ayz^4=0,&(1+a+a^3)yzzyx-2ayz^4=0, \\
(1+a^2)yzzyz-ayz^3y=0,&(1-a)yzzyz-(1-a+a^2)yz^3y=0, \\
(1-2a)yzzyy-a^2yz^3x=0,&(1+a^2)yzzyy-ayz^3x=0.
\end{array}
\end{equation}
The determinants of the matrices of the coefficients of the equations in the first and in the third lines of the above display equal to $a(1-3a-a^3)$ and therefore do not vanish according to (\ref{aneq}). The determinant arising from the second row is $1-2a+3a^2-a^3+a^4$ is also non-zero by (\ref{aneq}). Thus (\ref{gbup52}) can be rewritten as
\begin{equation}\label{gbup51b}
yzzyx=yzzyz=yzzyy=yz^4=yz^3y=yz^3x=0\ \ \text{in $A$}.
\end{equation}

Now by (\ref{AA5}), (\ref{gbup51a}) and (\ref{gbup51b}), $A_5=\{0\}$. Hence $H_A(t)=1+3t+6t^2+9t^3+9t^4$, which completes the proof.
\end{proof}

\begin{lemma}\label{gsk1} If $\K\neq \Z_2$, then there is $\alpha\in\K^6$ such that $H_A(t)=1+3t+6t^2+9t^3+9t^4$ and $H_{A^!}(t)=1+3t+3t^2$, where $A=\widehat{Q}^\xi$ with $\xi=(\alpha_1,\dots,\alpha_6,1,1,1)$.
\end{lemma}

\begin{proof} If ${\rm char}\,\K\in\{3,5\}$, the result follows from Lemma~\ref{PHEW0}. For the rest of the proof we assume that ${\rm char}\,\K\notin\{3,5\}$. By Lemma~\ref{PHEW}, it suffices to find $a\in\K$ for which (\ref{aneq}) is satisfied. If ${\rm char}\,\K\notin\{2,3,5\}$, then $a=-1$ satisfies (\ref{aneq}). Thus it remains to consider the case ${\rm char}\,\K=2$. In this case, one easily verifies that (\ref{aneq}) is equivalent to
\begin{equation}\label{aneqch2}
a\neq 0,\ a\neq 1,\ a^3+a+1\neq 0,\ a^5+a^3+1\neq 0\ \ \text{and}\ \ a^4+a^3+a^2+a+1\neq 0.
\end{equation}
The total number of $a$ failing (\ref{aneqch2}) can not  exceed $14$. Thus a required $a$ does exist provided $\K$ has more than 14 elements. This leaves us with two options to consider: $|\K|=4$ and $|\K|=8$. If $\K$ is the 4-element field, there is $a\in\K$ satisfying $a^2+a+1=0$. Such an $a$ also satisfies (\ref{aneqch2}). If $\K$ is the 8-element field, there is $a\in\K$ satisfying $a^3+a^2+1=0$. Again, such an $a$ satisfies (\ref{aneqch2}).
\end{proof}

\begin{lemma}\label{kz2} If $\K=\Z_2$, then $H_{\min}=1+3t+6t^2+9t^3+9t^4+5t^5+t^6$. Furthermore, there is $\alpha\in\K^6$  such that $H_A(t)=1+3t+6t^2+9t^3+9t^4+5t^5+t^6$ and $H_{A^!}(t)=1+3t+3t^2$, where $A=\widehat{Q}^\xi$ with $\xi=(\alpha_1,\dots,\alpha_6,1,1,1)$.
\end{lemma}

\begin{proof} First, let $\xi=(1,1,1,0,0,1,1,1,1)\in\Z_2^9$. Then the generalized Sklyanin algebra $A=\widehat{Q}^\xi$
is given be the relations $xx+yz$, $xy+zz$ and $xz+yy+zx$. A direct computation shows that these relations together with $yyx+yzz+zyz+zzy$, $yyz+yzz+zyz+zzx$, $yzx+zzz$, $yyyy+zyzz+zzyz$, $yzyx+zyzz$, $yzyy+yzzx+zzzz$, $yzyz+zzzx$, $yzzy+zzyy+zzzx$, $yzzz+zzzy$, $zyyy+zzyx$, $zyzzx$, $zzyzz$, $zzzyz$, $zzzzz$ and $zzzzyy$ form a Gr\"obner basis of the ideal of relations of $A$. The complete list of normal words is $x$, $y$, $z$, $yx$, $yy$, $yz$, $zx$, $zy$, $zz$,
$zyx$, $zzx$, $yyy$, $yzy$, $zyy$, $zzy$, $yzz$, $zyz$, $zzz$, $zzyx$, $yzzx$, $zzzx$, $zzyy$, $zyzy$, $zzzy$, $zzyz$,  $zyzz$, $zzzz$, $zzzyx$, $zzzzx$, $zzzyy$, $zzyzy$, $zzzzy$ and $zzzzyx$, which gives $H_A(t)=1+3t+6t^2+9t^3+9t^4+5t^5+t^6$.
The dual algebra $A^!$ is given by the relations $xx+yz$, $zy$, $xy+zz$, $yx$, $xz+zx$ and $yy+zx$. A direct computation shows that these relations together with $yzx$, $yzz$, $zzx$ and $zzz$ form a Gr\"obner basis of the ideal of relations of $A^!$.  The complete list of normal words is $x$, $y$, $z$, $yz$, $zx$ and $zz$, which gives $H_{A^!}(t)=1+3t+3t^2$.

Finally, we sketch the proof of the equality $H_{\min}=1+3t+6t^2+9t^3+9t^4+5t^5+t^6$ in the case $\K=\Z_2$. First, exactly as in the beginning of the proof of Theorem~\ref{COCOCO}, one shows that if $p_jq_kr_l=0$ for every $j,k,l$ satisfying $\{j,k,l\}=\{1,2,3\}$, then the Hilbert series of the corresponding generalized Sklyanin algebra is componentwise bigger than $1+3t+6t^2+9t^3+9t^4+5t^5+t^6$. Thus, up to a permutation of variables, we can assume that $p_3q_2r_1\neq 0$. In a similar manner, one checks that if $r_1r_2r_3=0$, then the Hilbert series of the corresponding generalized Sklyanin algebra is componentwise bigger than $1+3t+6t^2+9t^3+9t^4+5t^5+t^6$. Thus we can assume that $p_3=q_2=r_1=r_2=r_3=1$. Thus leads to $A$ defined by the relations $xx+ayz+bzy$, $xy+cyx+zz$, $xz+yy+dzx$ with $a,b,c,d\in\Z_2$. Again, in the case $a=b=0$, one easily checks that Hilbert series of $A$ is componentwise bigger than $1+3t+6t^2+9t^3+9t^4+5t^5+t^6$. Same goes for $a=b=c=d=1$ (in this case $A$ is a Sklyanin algebra). Now, swapping $y$ and $z$ corresponds to simultaneous swapping of $a$ and $b$ and of $c$ and $d$. Thus we are left with the following options for the quadruple $(a,b,c,d)$: $(1,1,0,0)$, $(1,1,0,1)$, $(1,1,1,1)$, $(1,0,0,0)$, $(1,0,0,1)$, $(1,0,1,0)$ and $(1,0,1,1)$. A direct computation shows that in all these cases $H_A$ is componentwise greater or equal to $1+3t+6t^2+9t^3+9t^4+5t^5+t^6$ and that the series $1+3t+6t^2+9t^3+9t^4+5t^5+t^6$ does occur (for instance, for $a=d=1$ and $b=c=0$). Thus $H_{\min}=1+3t+6t^2+9t^3+9t^4+5t^5+t^6$ in the case $\K=\Z_2$.
\end{proof}

\begin{proof}[Proof of Theorem~$\ref{copo22}$] The case $\K=\Z_2$ is covered by Lemma~\ref{kz2}. Assume now that $|\K|>2$. By Lemma~\ref{gsk1}, there is a generalized Sklyanin algebra $A$ satisfying $H_A(t)=1+3t+6t^2+9t^3+9t^4$ and $H_{A^!}=1+3t+3t^2$. Now let $B$ be an arbitrary generalized Sklyanin algebra. The Golod--Shafarevich theorem gives a lower estimate for the dimensions of the graded components of a quadratic algebra in terms of the numbers of generators and relations. In our case it yields $\dim B_2\geq 6$, $\dim B_3\geq 9$ and $\dim B_4\geq 9$. It immediately follows that $H_{\min}=1+3t+6t^2+9t^3+9t^4$, which completes the proof.
\end{proof}

\begin{proof}[Proof of Theorem~$\ref{copo02}$] For $\alpha\in \K^6$, let $\xi_\alpha=(\alpha_1,\dots,\alpha_6,1,1,1)\in\K^9$. Lemmas~\ref{kz2} and~\ref{gsk1} provide $\alpha\in\K^6$ such that the spaces $B_7$ and $B^!_3$ vanish, where $B=\widehat{Q}^{\xi_\alpha}$. By Lemma~\ref{minhs}, there is a non-empty Zarissky open subset $V$ of $\K^6$ such that $A_7=A^!_3=\{0\}$ for $A=\widehat{Q}^{\xi_\alpha}$ with $\alpha\in V$. Now let
$$
\textstyle U=\bigl\{\xi\in\K^9:\xi_1\xi_2\xi_3\neq 0,\ \bigl(\frac{\xi_4}{\xi_1},\frac{\xi_5}{\xi_
2},\frac{\xi_6}{\xi_3},\frac{\xi_7}{\xi_1},\frac{\xi_8}{\xi_2},\frac{\xi_9}{\xi_3}\bigr)\in V\bigr\}.
$$
Clearly, $U$ is non-empty and Zarissky open in $\K^9$ and $\{\widehat{Q}^{\xi_\alpha}:\alpha\in V\}=\{\widehat{Q}^{\xi}:\xi\in U\}$. Hence for $A=\widehat{Q}^{\xi}$ with $\xi\in U$ both $A$ and $A^!$ are finite
dimensional. This completes the proof of Theorem~$\ref{copo02}$.
\end{proof}

\vskip1truecm

{\bf Acknowledgements}

We are grateful to IHES and MPIM for hospitality, support, and excellent research atmosphere. We would like to thank M.Kontsevich and V.Sokolov for useful conversations. We appreciate very much referee's careful reading, which helped to improve the text. 
This work is funded by the ERC grant 320974, and partially supported by the project PUT9038.

\vskip1truecm
\normalsize

\rm

\normalsize
\vskip1truecm
\scshape

\noindent   Natalia Iyudu

\noindent School of Mathematics

\noindent  The University of Edinburgh

\noindent James Clerk Maxwell Building

\noindent The King's Buildings

\noindent Mayfield Road

\noindent Edinburgh

\noindent Scotland EH9 3JZ

\noindent E-mail address: \qquad {\tt niyudu@staffmail.ed.ac.uk}\ \ \

\vskip1truecm

\noindent    Stanislav Shkarin

\noindent Queens's University Belfast

\noindent Department of Pure Mathematics

\noindent University road, Belfast, BT7 1NN, UK

\noindent E-mail address:\qquad {\tt s.shkarin@qub.ac.uk}


\begin{thebibliography}{99}

\itemsep=-2pt

\bibitem{AS}M.~Artin and W.~Shelter, \it Graded algebras of global dimension $3$, \rm
Adv. in Math. \bf66\rm\ (1987), 171--216.

\bibitem{ATV1}M.~Artin, J.~Tate and M.~Van den Bergh, \it Modules over regular algebras of dimension 3 \rm Invent.Math. {\bf 106} (1991), 335--388.

\bibitem{ATV2}M.~Artin, J.~Tate and M.~Van den Bergh, \it Some algebras associated to automorphisms of elliptic curves. \rm The Grothendieck Festschrift, Vol. I, 33–-85, Progr. Math., {\bf 86}, Birkh\"auser Boston, Boston, MA, 1990.

\bibitem{BW}R.~Bocklandt, T.~Schedler and M.~Wemyss, \it Superpotentials and higher order derivations. \rm J. Pure Appl. Algebra {\bf 214} (2010), no. 9, 1501–-1522.

\bibitem{dr}V.~Drinfeld, \it On quadratic quasi-commutational relations in
quasi-classical limit, \rm Selecta Math. Sovietica {\bf 11} (1992), 317--326.

\bibitem{DV1}M.~Dubois-Violette, \it Graded algebras and multilinear forms. \rm C. R. Math. Acad. Sci. Paris 341 (2005), no. 12, 719-–724.

\bibitem{DV2}M.~Dubois-Violette, \it Multilinear forms and graded algebras. \rm J. Algebra 317(2007), no. 1, 198–-225.

\bibitem{Er}M.~Ershov, \it Golod–Shafarevich groups: A survey. \rm Int. J. Algebra Comput. {\bf 22}(2012), N5, 1–-68

\bibitem{G}V.~Ginzburg, \it Calabi-Yau algebras \rm ArXiV 0612139v3, 2007

\bibitem{gosh}E.~Golod and I.~Shafarevich, \it On the class field tower \rm
(Russian), Izv. Akad. Nauk SSSR Ser. Mat. \bf28\rm\ 1964, 261--272.

\bibitem{Sol}E.~Herscovich and A.~Solotar,  \it
Hochschild and cyclic homology of Yang-Mills algebras.
\rm J. Reine Angew. Math. {\rm 665} (2012), 73--156.

\bibitem{Ag1}T.~Lenagan and A.~Smoktunowicz, \it An infinite dimensional affine nil algebra with finite Gelfand-Kirillov dimension, \rm J. Amer. Math. Soc. {\bf 20} (2007), 989--1001.

\bibitem{Ko}M.~Kontsevich, \it Formal (non)commutative symplectic geometry. \rm The Gel'fand Mathematical Seminars, 1990–-1992, 173-–187, Birkhäuser Boston, Boston, MA, 1993.

\bibitem{ode}A.~Odesskii, \it Elliptic algebras, \rm  Russian Math. Surveys {\bf 57} (2002), 1127--1162

\bibitem{odf}A.~Odesskii and B.~Feigin, \it Sklyanin's elliptic algebras. (Russian) \rm Funktsional. Anal. i Prilozhen. {\bf 23} (1989), no. 3, 45--54; translation in Funct. Anal. Appl. {\bf 23} (1990), no. 3, 207-–214

\bibitem{popo}A.~Polishchuk and L.~Positselski, \it Quadratic
algebras, \rm University Lecture Series \bf37, \rm\ American
Mathematical Society, Providence, RI, 2005

\bibitem{S1}D.~Rogalski, S.~Sierra and T.~Stafford, \it Classifying orders in the Sklyanin algebra, \rm Algebra and Number Theory {\bf 9} (2015), 2055--2119

\bibitem{S2}D.~Rogalski, S.~Sierra and T.~Stafford, \it Noncommutative blowups of elliptic algebras, \rm Algebras and Representation Theory {\bf 18} (2015), 491--529

\bibitem{suri}V.~Sokolov, \it private communication, \rm IHES, January 2014.

\bibitem{skl}E.~Sklyanin, \it Some algebraic structures connected with the Yang-Baxterequation. Representations of a quantum algebra. (Russian)\rm Funktsional. Anal. i Prilozhen. {\bf 17} (1983), no. 4, 34–-48.

\bibitem{Ag2}A.~Smoktunowicz, \it A simple nil ring exists,\rm Communications in Algebra  {\it 30} (2002), no.1, 27--59.

\bibitem{Ag3}A.~Smoktunowicz, \it  Polynomial rings over nil rings need not be nil, \rm Journal of Algebra, {\bf 233} (2000), no.2, 427--436.

\bibitem{Ag4}A.~Smoktunowicz \it There are no graded domains with GK dimension strictly between 2 and 3, \rm Inventiones Mathematicae {\bf 164}, (2006), 635--640.

\bibitem{W}C.~Walton, \it Representation theory of three-dimensional Sklyanin algebras. \rm Nuclear Phys. B {\bf 860} (2012), no. 1, 167-–185.

\bibitem{Zelm}E.~Zelmanov, \it Some open problems in the theory of infinite dimensional algebras, \rm J. Korean Math. Soc. {\bf 44}(2007), N5, 1185-–1195.

\end{thebibliography}
\end{document}